%% file: momentIneq-generalized-2019.tex
\documentclass[11pt,preprint]{imsart}
\usepackage{amsmath,amsthm,verbatim,amssymb,amsfonts,amscd, graphicx}
\usepackage{amssymb,bm}
\usepackage{mathrsfs}
\usepackage{mathabx} % for symbol vvvert, tripe vertical line
\usepackage{hyperref}
\hypersetup{
    colorlinks,
    citecolor=black,
    filecolor=black,
    linkcolor=blue,
    urlcolor=black
}
\usepackage{url}
\usepackage{verbatim}
\usepackage{constants}
\usepackage{enumerate}
\newcommand{\rom}[1]{\uppercase\expandafter{\romannumeral #1\relax}}
\newcommand{\beas}{\begin{eqnarray*}}
\newcommand{\enas}{\end{eqnarray*}}
\newcommand{\bea}{\begin{eqnarray}}
\newcommand{\ena}{\end{eqnarray}}
\newcommand{\bms}{\begin{multline*}}
\newcommand{\ems}{\end{multline*}}
\newcommand{\bels}{\begin{align*}}
\newcommand{\enls}{\end{align*}}
\newcommand{\bel}{\begin{align}}
\newcommand{\enl}{\end{align}}

\newcommand{\ignore}[1]{}
\newcommand{\tr}{\mbox{tr}}

\newcommand{\mf}{\mathbf}
\newcommand{\wdt}{\widetilde}
\newtheorem{theorem}{Theorem}[section]
\newtheorem{corollary}{Corollary}[section]

\newtheorem{remark}{Remark}[section]
\newtheorem{lemma}{Lemma}[section]
\newtheorem{definition}{Definition}[section]

\newtheorem{example}{Example}

\def\blfootnote{\xdef\@thefnmark{}\@footnotetext}
\newcommand{\expect}[1]{\mathbb{E}{\l[#1\r]}}

\input{stas}

\begin{document}

\begin{frontmatter}
\title{\Large Moment inequalities for matrix-valued U-statistics of order 2}
\maketitle
\runtitle{Tail bound for matrix U-statistics}

\begin{aug}
\author{\fnms{Stanislav} \snm{Minsker}\thanksref{t2,t3}\ead[label=e1]{minsker@usc.edu}}
\and
\author{\fnms{Xiaohan} \snm{Wei}\thanksref{t1,t3}\ead[label=e3]{xiaohanw@usc.edu}}
\thankstext{t2}{Department of Mathematics, University of Southern California}
\thankstext{t1}{Department of Electrical Engineering, University of Southern California}
\thankstext{t3}{Authors gratefully acknowledge support by the National Science Foundation grant DMS-1712956.}
\runauthor{S. Minsker and X. Wei}
\end{aug}

\begin{abstract}
We present Rosenthal-type moment inequalities for matrix-valued U-statistics of order 2. 
As a corollary, we obtain new matrix concentration inequalities for U-statistics. 
One of our main technical tools, a version of the non-commutative Khintchine inequality for the spectral norm of the Rademacher chaos, could be of independent interest. 
%We show that the moment bound is tight when the U-statistics is of order 2, in which case a matching lower bound can be obtained. Underlying both of these inequalities is a tight expected operator norm bound for
%matrix Rademacher chaos with a weak logarithm dependency on the dimension and a matching lower bound. Based on the two inequalities, we obtain moment and tail bounds for non-degenerated matrix U-statistics of arbitrary order. 
\end{abstract}
\end{frontmatter}

%#######################################
\section{Introduction.}
\label{section:intro}
%#######################################

Since being introduced by W. Hoeffding \cite{hoeffding1948class}, U-statistics have become an active topic of research. 
Many classical results in estimation and testing are related to U-statistics; detailed treatment of the subject can be found in excellent monographs \cite{Decoupling, korolyuk2013theory,serfling2009approximation,kowalski2008modern}. 
A large body of research has been devoted to understanding the asymptotic behavior of real-valued U-statistics. 
Such asymptotic results, as well as moment and concentration inequalities, are discussed in the works \cite{PM-decoupling-1995,Decoupling,gine2000exponential,gine1992hoffmann,ibragimov1999analogues, gine2001lil,houdre2003exponential}, among others. 
The case of vector-valued and matrix-valued U-statistics received less attention; natural examples of matrix-valued U-statistics include various estimators of covariance matrices, such as the usual sample covariance matrix and the estimators based on Kendall's tau \cite{wegkamp2016adaptive,han2017statistical}. 

Exponential and moment inequalities for Hilbert space-valued U-statistics have been developed in \cite{adamczak2008lil}. 
The goal of the present work is to obtain moment and concentration inequalities for generalized degenerate U-statistics of order 2 with values in the set of matrices with complex-valued entries equipped with the operator (spectral) norm. 
The emphasis is made on expressing the upper bounds in terms of \emph{computable} parameters. 
Our results extend the matrix Rosenthal's inequality for the sums of independent random matrices due to Chen, Gittens and Tropp \cite{chen2012masked} (see also \cite{junge2013noncommutative,mackey2014matrix}) to the framework of U-statistics. 
As a corollary of our bounds, we deduce a variant of the Matrix Bernstein inequality for U-statistics of order 2. 

We also discuss connections of our bounds with general moment inequalities for Banach space-valued U-statistics due to R. Adamczak \cite{adamczak2006moment}, and leverage Adamczak's inequalities to obtain additional refinements and improvements of the results.

We note that U-statistics with values in the set of self-adjoint matrices have been considered in \cite{chen2016bootstrap}, however, most results in that work deal with the element-wise sup-norm, while we are primarily interested in results about the moments and tail behavior of the spectral norm of U-statistics. 
Another recent work \cite{minsker2018robust} investigates robust estimators of covariance matrices based on U-statistics, but deals only with the case of non-degenerate U-statitistics that can be reduced to the study of independent sums.

The key technical tool used in our arguments is the extension of the non-commutative Khintchine's inequality (Lemma \ref{bound-on-expectation}) which could be of independent interest.

%#######################################
\section{Notation and background material.}
%#######################################

Given $A\in \mb C^{d_1\times d_2}$, $A^\ast\in \mb C^{d_2\times d_1}$ will denote the Hermitian adjoint of $A$. 
$\mb H^d\subset \mb C^{d\times d}$ stands for the set of all self-adjoint matrices. 
If $A=A^\ast$, we will write $\lambda_{\mx}(A)$ and $\lambda_{\mn}(A)$ for the largest and smallest eigenvalues of $A$. 

Everywhere below, $\|\cdot\|$ stands for the spectral norm $\|A\|:=\sqrt{\lambda_{\mx}(A^\ast A)}$. 
If $d_1=d_2=d$, we denote by $\tr (A)$ the trace of $A$. 
The Schatten p-norm of a matrix $A$ is defined as $\|A\|_{S_p} = \l( \tr (A^\ast A)^{p/2} \r)^{1/p}.$
When $p=1$, the resulting norm is called the nuclear norm and will be denoted by $\|\cdot\|_\ast$. 
The Schatten 2-norm is also referred to as the Frobenius norm or the Hilbert-Schmidt norm, and is denoted by $\|\cdot\|_{\mathrm{F}}$; and the associated inner product is 
$\dotp{A_1}{A_2}=\tr(A_1^\ast A_2)$. 

Given $z\in \mb C^d$, $\l\| z \r\|_2=\sqrt{z^\ast z}$ stands for the usual Euclidean norm of $z$. 
Let $A,~B\in\mathbb{H}^{d}$. 
We will write $A\succeq B  (\textrm{or}A\succ B)$ iff $A-B$ is nonnegative (or positive) definite.
For $a,b\in \mb R$, we set $a\vee b:=\max(a,b)$ and $a\wedge b:=\min(a,b)$. 
We use $C$ to denote absolute constants that can take different values in various places.  

Finally, we introduce the so-called Hermitian dilation which is a tool that often allows to reduce the problems involving general rectangular matrices to the case of Hermitian matrices. 
\begin{definition}
Given a rectangular matrix $A\in\mb C^{d_1\times d_2}$, the Hermitian dilation $\m D: \mb C^{d_1\times d_2}\mapsto \mb C^{(d_1+d_2)\times (d_1+d_2)}$ is defined as
\begin{align}
\label{eq:dilation}
&
\m D(A)=\begin{pmatrix}
0 & A \\
A^\ast & 0
\end{pmatrix}.
\end{align}
\end{definition}
\noindent Since 
$\m D(A)^2=\begin{pmatrix}
A A^\ast & 0 \\
0 & A^\ast A
\end{pmatrix},$ 
it is easy to see that $\| \m D(A) \|=\|A\|$.

The rest of the paper is organized as follows. 
Section \ref{sec:setup} contains the necessary background on U-statistics. 
Section \ref{sec:moment} contains our main results -- bounds on the $\mb H^d$-valued Rademacher chaos and moment inequalities for $\mb H^d$-valued U-statistics of order 2. 
Section \ref{section:adamczak} provides comparison of our bounds to relevant results in the literature, and discusses further improvements. 
Finally, Section \ref{sec:proof} contains the technical background and proofs of the main results.

%#############################
\subsection{Background on U-statistics.}
\label{sec:setup}
%#############################

Consider a sequence of i.i.d. random variables $X_1,\ldots,X_n$ ($n\geq2$) taking values in a measurable space 
$(\m S,\mathcal{B})$, and let $P$ denote the distribution of $X_1$. 
Define 
\[
I_n^m:=\{(i_1,\ldots,i_m):~1\leq i_j\leq n,,~i_j\neq i_k~\textrm{if}~j\neq k\},
\] 
and assume that $H_{i_1,\ldots,i_m}: \m S^m\rightarrow\mb H^d$, $(i_1,\ldots,i_m)\in I_n^m$, $2\leq m\leq n$, are $\mathcal{S}^m$-measurable, permutation-symmetric kernels, meaning that $H_{i_1,\ldots,i_m}(x_1,\ldots,x_m) = H_{i_{\pi_1},\ldots,i_{\pi_m}}(x_{\pi_1},\ldots,x_{\pi_m})$ for any $(x_1,\ldots,x_m)\in \m S^m$ and any permutation $\pi$. For example, when $m=2$, this conditions reads as $H_{i_1,i_2}(x_1,x_2) = H_{i_2,i_1}(x_2,x_1)$ for all $i_1\ne i_2$ and $x_1,x_2$. 
The generalized U-statistic is defined as \cite{Decoupling}
\begin{equation}
\label{u-stat}
U_n:=\sum_{(i_1,\ldots,i_m)\in I_n^m}H_{i_1,\ldots,i_m}(X_{i_1},\ldots,X_{i_m}).
\end{equation} 
When $H_{i_1,\ldots,i_m}\equiv H$, we obtain the classical U-statistics. 
It is often easier to work with the decoupled version of $U_n$ defined as 
\[
U'_n =  \sum_{(i_1,\ldots, i_m)\in I_n^m} H_{i_1,\ldots,i_m}\l( X^{(1)}_{i_1},\ldots,X^{(m)}_{i_m} \r),  
\]
where $\l\{ X_i^{(k)} \r\}_{i=1}^n, \ k=1,\ldots,m$ are independent copies of the sequence $X_1,\ldots,X_n$. 
Our ultimate goal is to obtain the moment and deviation bounds for the random variable $\| U_n  - \mb E U_n \|$.   
 
Next, we recall several useful facts about U-statistics.  
The projection operator $\pi_{m,k}~(k\leq m)$ is defined as 
\[
\pi_{m,k}H(\mathbf{x}_{i_1},\ldots,\mathbf{x}_{i_k})
:= (\delta_{\mathbf{x}_{i_1}}- P)\ldots(\delta_{\mathbf{x}_{i_k}} - P) P^{m-k}H,
\]
where 
\[
\mathcal{Q}^mH := \int\ldots\int H(\mathbf{y}_1,\ldots,\mathbf{y}_m)dQ(\mathbf{y}_1)\ldots dQ(\mathbf{y}_m),
\] 
for any probability measure $Q$ on $(\m S,\mathcal{B})$, and $\delta_{x}$ is a Dirac measure concentrated at $x\in \m S$. 
For example, $\pi_{m,1}H(x) = \mb E \l[ H(X_1,\ldots,X_m)| X_1=x\r] - \mb E H(X_1,\ldots,X_m)$.
\begin{definition}
Let $F: \m S^m\rightarrow\mb H^d$ be a measurable function. We will say that $F$ is $P$-degenerate of order $r$ 
($1\leq r<m$) iff 
\[
\mb EF(\mathbf{x}_1,\ldots,\mathbf{x}_r,X_{r+1},\ldots,X_m)=0~\forall \mathbf{x}_1,\ldots,\mathbf{x}_r\in \m S,
\]
and $\mb E F(\mathbf{x}_1,\ldots,\mathbf{x}_r, \mathbf{x}_{r+1},X_{r+2},\ldots,X_m)$ is not a constant function. 
Otherwise, $F$ is non-degenerate.
\end{definition}
For instance, it is easy to check that $\pi_{m,k}H$ is degenerate of order $k-1$.  
If $F$ is degenerate of order $m-1$, then it is called \emph{completely degenerate}. 
From now on, we will only consider generalized U-statistics of order $m=2$ with completely degenerate (that is, degenerate of order 1) kernels. 
The case of non-degenerate U-statistics is easily reduced to the degenerate case via the \emph{Hoeffding's decomposition;} see page 137 in \cite{Decoupling} for the details.

%#####################################
\section{Main results.}
\label{sec:moment}
%#####################################

Rosenthal-type moment inequalities for sums of independent matrices 
have appeared in a number of previous works, including \cite{chen2012masked, mackey2014matrix, JT-matrix}. For example, the following inequality follows from Theorem A.1 in \cite{chen2012masked}:
\begin{lemma}[Matrix Rosenthal inequality]
\label{lemma:rosenthal-1}
Suppose that $q \geq 1$ is an integer and fix $r\geq q\vee \log d$. Consider a finite sequence of $\{\mathbf{Y}_i\}$ of independent $\mb H^d$-valued random matrices. Then
\begin{multline}
\label{iid-rosenthal}
\left(\mb E\l\| \sum_i \l( \mathbf{Y}_i -\mb E \mathbf{Y}_i \r)\r\|^{2q} \right)^{1/2q}\leq 2\sqrt{er}\l\|  \l(\sum_i \mb E 
\l( \mathbf{Y}_i - \mb E\mathbf{Y}_i \r)^2 \r)^{1/2}  \r\| 
\\
+4\sqrt{2} er\l( \mb E\max_i\|\mathbf{Y}_i - \mb E\mathbf{Y}_i \|^{2q} \r)^{1/2q}.
\end{multline}
\end{lemma}
The bound above improves upon the moment inequality that follows from the matrix Bernstein's inequality (see Theorem 1.6.2 in \cite{JT-matrix}):
% for the bound with explicit constants):
\begin{lemma}[Matrix Bernstein's inequality]
\label{bernstein}
Consider a finite sequence of $\{\mathbf{Y}_i\}$ of independent $\mb H^d$-valued random matrices such that 
$\|\mathbf{Y}_i - \mb E\mathbf{Y}_i\|\leq B$ almost surely. Then 
\[
\Pr\l( \l\|  \sum_i \l( \mathbf{Y}_i -\mb E\mathbf{Y}_i \r)  \r\| \geq 2\sigma\sqrt{u} + \frac{4}{3}Bu \r)\leq 2de^{-u},
\]
where $\sigma^2:= \l\| \sum_i \mb E\l(\mathbf{Y}_i -\mb E\mathbf{Y}_i\r)^2 \r\|$. 
%and $C_1>0$ is an absolute constant.
\end{lemma}
\noindent Indeed, Lemma \ref{tail-to-moment} implies, with $a_0 = C\l(\sigma\sqrt{\log(2d)} +B\log (2d)\r)$ for some absolute constant $C>0$ and after some simple algebra, that
\[
\l( \mb E\l\|  \sum_i \l( \mathbf{Y}_i - \mb E\mathbf{Y}_i \r) \r\|^q \r)^{1/q}\leq C_2\l( \sqrt{q+\log(2d)}\,\sigma + (q+\log(2d))B \r),
\]
for an absolute constant $C_2>0$ and all $q\geq1$. 
This bound is weaker than \eqref{iid-rosenthal} as it requires almost sure boundedness of $\|\mathbf{Y}_i - \mb E \mathbf{Y}_i\|$ for all $i$. 
One the the main goals of this work is to obtain operator norm bounds similar to inequality \eqref{iid-rosenthal} for $\mb H^d$-valued U-statistics of order 2.

%########################################
\subsection{Degenerate U-statistics of order 2.}
%########################################

Moment bounds for scalar U-statistics are well-known, see for example the work \cite{gine2000exponential} and references therein. 
Moreover, in \cite{adamczak2006moment}, author obtained moment inequalities for general Banach-space valued U-statistics. 
%we discuss them in the context of our work in Section \ref{section:adamczak}. 
Here, we aim at improving these bounds for the special case of $\mb H^d$-valued U-statistics of order 2. 
We discuss connections and provide comparison of our results with the bounds obtained by R. Adamczak \cite{adamczak2006moment} in Section \ref{section:adamczak}.

%##############################
\subsection{Matrix Rademacher chaos.}
%##############################

The starting point of our investigation is a moment bound for the matrix Rademacher chaos of order 2. 
%with an optimal dependence of the dimension factor. 
This bound generalizes the spectral norm inequality for the matrix Rademacher series, see \cite{JT-matrix, tropp2016expected, tropp2016second,random-matrix-2010}.
We recall Khintchine's inequality for the matrix Rademacher series for the ease of comparison: let $A_1,\ldots,A_n\in\mathbb{H}^{d}$ be a sequence of fixed matrices, and $\varepsilon_1,\ldots,\varepsilon_n$ -- a sequence of i.i.d. Rademacher random variables. Then
\begin{align}
\label{matrix-Khintchine}
&
\l( \mb E\l\|\sum_{i=1}^n\varepsilon_i A_i\r\|^2 \r)^{1/2}\leq \sqrt{e(1+2\log d)}\cdot
\l\| \sum_{i=1}^n A_i^2 \r\|^{1/2}.
\end{align}
Furthermore, Jensen's inequality implies this bound is tight (up to a logarithmic factor). 
Note that the expected norm of $\sum \eps_i A_i$ is controlled by the single ``matrix variance'' parameter  $\l\| \sum_{i=1}^n A_i^2 \r\|$. 
Next, we state the main result of this section, the analogue of inequality \eqref{matrix-Khintchine} for the Rademacher chaos of order 2.
\begin{lemma}
\label{bound-on-expectation}
Let $\{A_{i_1,i_2}\}_{i_1,i_2=1}^{n}\in \mb H^d$ be a sequence of fixed matrices.
Assume that $\l\{\varepsilon_j^{(i)}\r\}_{j\in\mathbb{N}}, \ i=1,2,$ are two independent sequences of i.i.d. Rademacher random variables, and define
\[
X= \sum_{(i_1,i_2)\in I_n^2} A_{i_1,i_2}\varepsilon_{i_1}^{(1)}\varepsilon_{i_2}^{(2)}.
\]
Then for any $q\geq1$,
\begin{multline}
\label{eq:khintchine}
\max\l\{\l\|GG^\ast \r\|,  \l\| \sum_{(i_1,i_2)\in I_n^2}A_{i_1,i_2}^2 \r\| \r\}^{1/2}
\leq \l( \mb E\|X\|^{2q}\r)^{1/(2q)} \\
\leq
\frac{4}{\sqrt{e}}\cdot r\cdot\max\l\{  \l\| GG^\ast \r\|, \l\| \sum_{(i_1,i_2)\in I_n^2}A_{i_1,i_2}^2 \r\| \r\}^{1/2},
\end{multline}
where $r:= q \vee\log d$, and the matrix $G\in\mathbb{H}^{nd}$ is defined via its block structure as
\begin{equation}
\label{eq:G}
G:=
\left(
\begin{array}{cccc}
0  & A_{1,2}  & \ldots & A_{1,n}  \\
A_{2,1}  & 0  & \ldots & A_{2,n}  \\
\vdots  & \vdots  &   \ddots & \vdots \\
A_{n,1} & A_{n,2} & \ldots & 0 
\end{array}
\right).
\end{equation}
\end{lemma}

\begin{remark}[Constants in Lemma \ref{bound-on-expectation}]
\label{remark:log-constant}
Matrix Rademacher chaos of order 2 has been studied previously in \cite{rauhut2009circulant}, \cite{pisier1998non} and \cite{HR-CS}, where Schatten-$p$ norm upper bounds were obtained by iterating Khintchine's inequality for Rademacher series. Specifically, the following bound holds for all $p\geq 1$ (see Lemma \ref{k-inequality-2} for the details):
\begin{align*}
&
\mb E\left\| X \right\|_{S_{2p}}^{2p}
\leq
2\l( \frac{2\sqrt 2}{e} p\r)^{2p}
\max\l\{  \l\|  \l(G  G^\ast \r)^{1/2}\r\|_{S_{2p}}^{2p}, \left\|\l(\sum_{i_1,i_2=1}^n A_{i_1,i_2}^2\r)^{1/2}\right\|_{S_{2p}}^{2p} \r\}.
\end{align*}
Using the fact that for any $B\in\mathbb{H}^{d}$, $\|B\|\leq\|B\|_{S_{2p}}\leq d^{1/2p}\|B\|$
and taking $p=q\vee \log(nd)$, one could obtain a ``na\"{i}ve'' extension of the inequality above, namely
\begin{align*}
&
\l( \mb E\|X\|^{2q}\r)^{1/(2q)}\leq
C\max\l( q,\log(nd) \r) \max\l\{  \l\| GG^\ast \r\|, \left\|\sum_{(i_1,i_2)\in I^n_2}A_{i_1,i_2}^2\right\| \r\}^{1/2}
\end{align*}
that contains an extra $\log(n)$ factor which is removed in Lemma \ref{bound-on-expectation}. 
\end{remark}

One may wonder if the term $\l\| GG^\ast \r\|$ in Lemma \ref{bound-on-expectation} is redundant. 
For instance, in the case when $\{A_{i_1,i_2}\}_{i_1,i_2}$ are scalars, it is easy to see 
$\l\| \sum_{(i_1,i_2)\in I_n^2 } A_{i_1,i_2}^2\r\|\geq\l\| GG^\ast \r\|$. 
However, a more careful examination shows that there is no strict dominance among $\l\| GG^\ast \r\|$ and 
$\l\| \sum_{(i_1,i_2)\in I_n^2 } A_{i_1,i_2}^2\r\|$. 
The following example presents a situation where $\l\| \sum_{(i_1,i_2)\in I_n^2 } A_{i_1,i_2}^2\r\|<\l\| GG^\ast \r\|$.
\begin{example}
\label{example:01}
Assume that $d~\geq~n~\geq~2$, let 
$\{\mathbf{a}_1,\ldots,\mathbf{a}_d\}$ be any orthonormal basis in $\mathbb{R}^d$, 
and $\mathbf{a}:=[\mathbf{a}_1^T,\ldots,\mathbf{a}_n^T]^T \in \mb R^{nd}$ be the ``vertical concatenation'' of $\mathbf{a}_1,\ldots,\mathbf{a}_d$. 
Define
\[
A_{i_1,i_2} := \mathbf{a}_{i_1}\mathbf{a}_{i_2}^T + \mathbf{a}_{i_2}\mathbf{a}_{i_1}^T,
~~i_1,i_2\in\{1,2,\ldots,n\},
\]
and
\[
X:= \sum_{(i_1,i_2)\in I_n^2} \varepsilon_{i_1}^{(1)}\varepsilon_{i_2}^{(2)}
A_{i_1,i_2}.
\]
Then $\l\| GG^\ast \r\|  = \l\| GG^T \r\| \geq (n-2) \| \mathbf{a}\|_{2}^2 = (n-2)n$, and 
$\l\|  \sum_{(i_1,i_2)\in I_n^2}  A_{i_1,i_2}^2 \r\| = 2(n-1)$. 
Details are outlined in Section \ref{section:example-proof}.
\end{example}
\noindent It follows from Lemma \ref{lemma:block-matrix} that 
\begin{align}
\label{eq:useful-bound}
&
\l\| GG^\ast \r\| \leq \sum_{i_1}\l\| \sum_{i_2: i_2\ne i_1} A^2_{i_1,i_2}\r\|.
\end{align}
Often, this inequality yields a ``computable'' upper bound for the right-hand side of the inequality \eqref{eq:khintchine}, however, in some cases it results in the loss of precision, as the following example demonstrates.
\begin{example}
\label{example:02}
Assume that $n$ is even, $d~\geq~n~\geq~2$, let 
$\{\mathbf{a}_1,\ldots,\mathbf{a}_d\}$ be an orthonormal basis in $\mathbb{R}^d$, and let $\m C\in \mb R^{n\times n}$ be an orthogonal matrix with entries $c_{i,j}$ such that $c_{i,i}=0$ for all $i$. Define 
\[
A_{i_1,i_2}= c_{i_1,i_2}\l(\mathbf{a}_{i_1}\mathbf{a}_{i_2}^T + \mathbf{a}_{i_2}\mathbf{a}_{i_1}^T \r),
~~i_1,i_2\in\{1,2,\ldots,n\},
\]
and $X:= \sum_{(i_1,i_2)\in I_n^2} \varepsilon_{i_1}^{(1)}\varepsilon_{i_2}^{(2)} A_{i_1,i_2}.$ 
Then $\l\| GG^\ast \r\| = 1$, $\l\|  \sum_{(i_1,i_2)\in I_n^2}  A_{i_1,i_2}^2 \r\| = 2$, but 
\[
\sum_{i_1} \left\| \sum_{i_2:i_2\neq i_1}A_{i_1,i_2}^2\right\| = n.
\] 
Details are outlined in Section \ref{section:example-proof}.
\end{example}

%##########################################################
\subsection{Moment inequalities for degenerate U-statistics of order 2.}
\label{sec:main-results}
%##########################################################

Let $H_{i_1,i_2}:\mathcal{S}\times \mathcal{S}\mapsto \mb H^d$, $(i_1,i_2)\in I_n^2$, be a sequence of degenerate kernels, for example, $H_{i_1,i_2}(x_1,x_2)=\pi_{2,2}\widehat H_{i_1,i_2}(x_1,x_2)$ for some non-degenerate permutation-symmetric $\widehat H_{i_1,i_2}$. 
Recall that $U_n$, the generalized U-statistic of order 2, has the form 
\[
U_n:=\sum_{(i_1,i_2)\in I_n^2}H_{i_1,i_2}(X_{i_1},X_{i_2}).
\]
Everywhere below, $\mb E_j[\cdot], \ j=1,2,$ stands for the expectation with respect to $\l\{X_i^{(j)}\r\}_{i=1}^n$ only (that is, conditionally on all other random variables). 
The following Theorem is our most general result; it can be used as a starting point to derive more refined bounds. 
\begin{theorem}
\label{thm:degen-moment}
Let $\l\{X_i^{(j)}\r\}_{i=1}^n, \ j=1,2,$ be $\mathcal{S}$-valued i.i.d. random variables, $H_{i,j}:\mathcal{S}\times \mathcal{S}\mapsto \mb H^d$ -- permutation-symmetric degenerate kernels. Then for all $q \geq 1$ and $r=\max(q ,\log (ed))$,
\begin{align*}
\l(\mb E \l\| U_n \r\|^{2q}  \r)^{1/2q}
\leq  &
4 \l(\mb E \l\| \sum_{(i_1,i_2)\in I^2_n} H_{i_1,i_2}\l(X_{i_1}^{(1)},X_{i_2}^{(2)}\r) \r\|^{2q}\r)^{1/2q} 
\\ 
\leq &
128/\sqrt{e} \Bigg[
16 r^{3/2} \l( \mb E\max_{i_1} \l\| \sum_{i_2:i_2\ne i_1 } H^2_{i_1,i_2} \l(X_{i_1}^{(1)},X_{i_2}^{(2)}\r) \r\|^q \r)^{1/(2q)} 
\\
% use that (2\sqrt{2e} + 8 \sqrt{2})  \leq 16
&
+ r  \l\| \sum_{(i_1,i_2)\in I_n^2 } \mb E H^2_{i_1,i_2}\l(X_{i_1}^{(1)},X_{i_2}^{(2)}\r) \r\|^{1/2}  
+ r\l( \mb E\l\|   \mb E_2 \wdt G \wdt G^\ast \r\|^{q}\r)^{1/2q} \Bigg],
\end{align*}
where the matrix $\wdt G\in \mb H^{n d}$ is defined as
\begin{align}
\label{eq:g}
&
\wdt G:=
\left(
\begin{array}{cccc}
0  & H_{1,2}\l(X_{1}^{(1)}, X_{2}^{(2)} \r)  & \ldots & H_{1,n}\l(X_{1}^{(1)}, X_{n}^{(2)} \r)  \\
H_{2,1}\l(X_{2}^{(1)}, X_{1}^{(2)} \r)  & 0  & \ldots & H_{2,n}\l(X_{2}^{(1)}, X_{n}^{(2)} \r)  \\
\vdots  & \vdots  &   \ddots & \vdots \\
H_{n,1}\l(X_{n}^{(1)}, X_{1}^{(2)} \r) & H_{n,2}\l(X_{n}^{(1)}, X_{2}^{(2)} \r) & \ldots & 0 
\end{array}
\right).
\end{align}
%and $A_{i_1,i_2} := H\l(X_{i_1}^{(1)}, X_{i_2}^{(2)} \r)$.
\end{theorem}
\begin{proof}
See Section \ref{proof:degen-moment}. 
\end{proof}

The following lower bound (proven in Section \ref{proof:lower}) demonstrates that all the terms in the bound of Theorem \ref{thm:degen-moment} are necessary.
\begin{lemma}
\label{lemma:degen-lower-bound}
Under the assumptions of Theorem \ref{thm:degen-moment},
\begin{multline*}
\l(\mb E \l\| U_n \r\|^{2q}  \r)^{1/2q} \geq 
C \l[
\l( \mb E\max_{i_1} \l\| \sum_{i_2:i_2\ne i_1} H^2_{i_1,i_2}\l(X_{i_1}^{(1)},X_{i_2}^{(2)}\r) \r\|^q \r)^{1/(2q)} \r.
\\
\l.+   \l( \mb E\l\| \mb E_2 \wdt G \wdt G^\ast \r\|^{q}\r)^{1/2q} + 
\l( \mb E\l\| \sum_{(i_1,i_2)\in I_n^2 } \mb E_2 H^2_{i_1,i_2}\l(X_{i_1}^{(1)},X_{i_2}^{(2)}\r) \r\|^{q} \r)^{1/2q}\r]
\end{multline*}
where $C>0$ is an absolute constant.
\end{lemma}
\begin{example}
Let $\{ A_{i_1,i_2} \}_{1\leq i_1<i_2\leq n}$ be fixed elements of $\mb H^d$ and $X_1,\ldots,X_n$ -- centered i.i.d. real-valued random variables such that $\var(X_1)=1$. Consider $\mathbf{Y}:=\sum_{i_1\ne i_2} A_{i_1,i_2} X_{i_1}X_{i_2}$, where $A_{i_2,i_1}=A_{i_1,i_2}$ for $i_2>i_1$. 
We will apply Theorem \ref{thm:degen-moment} to obtain the bounds for $\l( \mb E\|\mathbf{Y}\|^{2q} \r)^{1/2q}$. In this case, $H_{i_1,i_2} \l(X_{i_1}^{(1)},X_{i_2}^{(2)}\r) = A_{i_1,i_2}X_{i_1}^{(1)} X_{i_2}^{(2)}$, and it is easy to see that
\[
\l( \mb E\max_{i_1} \l\| \sum_{i_2:i_2\ne i_1 } H^2_{i_1,i_2} \l(X_{i_1}^{(1)},X_{i_2}^{(2)}\r) \r\|^q \r)^{1/(2q)} \leq 
\max_i \l\| \sum_{j\ne i} A^2_{i,j}\r\|^{1/2} \l( \mb E\max_{1\leq i\leq n} |X_i|^{2q}  \r)^{1/q}
\]
and $ \l\| \sum_{(i_1,i_2)\in I_n^2 } \mb E H^2_{i_1,i_2}\l(X_{i_1}^{(1)},X_{i_2}^{(2)}\r) \r\|^{1/2} = \l\| \sum_{(i_1,i_2)\in I_n^2} A^2_{i_1,i_2}\r\|^{1/2}$. Moreover,  
\[
\l( \mb E_2 \wdt G \wdt G^\ast \r)_{i,j} = X_i^{(1)} X_j^{(1)}\sum_{k\ne i,j} A_{i,k} A_{j,k},
\] 
implying that $\mb E_2 \wdt G \wdt G^\ast  = D \, G \, D$, where $G$ is defined as in \eqref{eq:G} and $D\in \mb H^{n d}$ is a diagonal matrix $D = \mathrm{diag}(X_1^{(1)},\ldots,X_n^{(1)}) \otimes I_d$, where $\otimes$ denotes the Kronecker product. 
It yields that $\l\| \mb E_2 \wdt G \wdt G^\ast \r\| \leq \max_i \l| X_i^{(1)} \r|^2 \cdot \l\| G G^\ast\r\|$, hence 
\[
\l( \mb E\l\|  \mb E_2 \wdt G \wdt G^\ast \r\|^{q}\r)^{1/2q} \leq \l\| G G^\ast\r\|^{1/2} \l( \mb E\max_{1\leq i\leq n} |X_i|^{2q}  \r)^{1/2q}.
\]
Combining the inequalities above, we deduce from Theorem \ref{thm:degen-moment} that 
\begin{multline}
\label{eq:polynomial}
\l( \mb E\|\mathbf{Y}\|^{2q} \r)^{1/2q} \leq C\Bigg[ r \l( \l\| \sum_{(i_1,i_2)\in I_n^2} A^2_{i_1,i_2}\r\|^{1/2} + \l( \mb E\max_{1\leq i\leq n} |X_i|^{2q}  \r)^{1/2q} \l\| GG^\ast \r\|^{1/2} \r) \\
+ r^{3/2} \max_{i_1} \l\| \sum_{i_2\ne i_1} A^2_{i_1,i_2}\r\|^{1/2} \l( \mb E\max_{1\leq i\leq n} |X_i|^{2q}  \r)^{1/q} \Bigg],
\end{multline}
where $r = \max(q,\log (ed))$. If for instance $|X_1|\leq M$ almost surely for some $M\geq 1$, it follows that 
\begin{equation*}
\l( \mb E\|\mathbf{Y}\|^{2q} \r)^{1/2q} \leq C\Bigg[ r \l( M \, \l\| GG^\ast \r\|^{1/2} + \l\| \sum_{(i_1,i_2)\in I_n^2} A^2_{i_1,i_2}\r\|^{1/2} \r)
+ r^{3/2} M^2 \max_{i_1} \l\| \sum_{i_2\ne i_1} A^2_{i_1,i_2}\r\|^{1/2} \Bigg].
\end{equation*}
On the other hand, if $X_1$ is not bounded but is sub-Gaussian, meaning that $\l(\mb E|X_1|^q \r)^{1/q} \leq C \sigma \sqrt{q}$ for all $q\in \mb N$ and some $\sigma>0$, then it is easy to check that 
\[
\l(\mb E\max_{1\leq i\leq n} |X_i|^{2q}\r)^{1/2q} \leq C_1 \sqrt{\log(n)} \sigma \sqrt{2q},
\]
and the estimate for $\l(\mb E \l\|\mathbf{Y} \r\|^{2q}\r)^{1/2q}$ follows from \eqref{eq:polynomial}.
\end{example}

Our next goal is to obtain more ``user-friendly'' versions of the upper bound, and we first focus on the term 
$ \mb E \big\| \mb E_2 \wdt G \wdt G^\ast \big\|^q$ appearing in Theorem \ref{thm:degen-moment} that might be difficult to deal with directly.  
It is easy to see that the $(i,j)$-th block of the matrix $\mb E_2 \wdt G\wdt G^\ast$ is 
\[
\l( \mb E_2 \wdt G \wdt G^\ast\r)_{i,j} = \sum_{k\ne i,j} \mb E_2 \l[ H_{i,k}(X_i^{(1)},X_k^{(2)}) H_{j,k}(X_j^{(1)},X_k^{(2)})\r].
\]
It follows from Lemma \ref{lemma:block-matrix} that
\begin{align}
\label{eq:d20}
\| \mb E_2 \wdt G \wdt G^\ast \| & \leq \sum_{i}  \l\| \l( \mb E_2 \wdt G \wdt G^\ast\r)_{i,i}\r\| = 
\sum_{i_1}  \l\| \sum_{i_2:i_2\ne i_1}\mb E_2 H_{i_1,i_2}^2\l(X^{(1)}_{i_1}, X_{i_2}^{(2)}\r) \r\|, 
\end{align}
hence 
\begin{multline*}
\l( \mb E\l\| \mb E_2\wdt G \wdt G^\ast \r\|^{q}\r)^{1/2q}\leq 
\l( \mb E \l( \sum_{i_1} \l\| \sum_{i_2:i_2\ne i_1} \mb E_2 H_{i_1,i_2}^2\l(X^{(1)}_{i_1}, X_{i_2}^{(2)}\r) \r\| \r)^q \r)^{1/2q} 
\\ 
\leq \l( \sum_{i_1}\mb E\l\|  \sum_{i_2:i_2\ne i_1} \mb E_2 H_{i_1,i_2}^2\l( X^{(1)}_{i_1}, X_{i_2}^{(2)}\r)\r\| \r)^{1/2}  \\
+ 2\sqrt{2eq} \l( \mb E\max_{i_1} \l\| \sum_{i_2:i_2\ne i_1} \mb E_2 H_{i_1,i_2}^2\l( X^{(1)}_{i_1}, X_{i_2}^{(2)}\r)\r\|^q \r)^{1/2q},
\end{multline*}
where we used Rosenthal's inequality (Lemma \ref{lemma:rosenthal-pd} applied with $d=1$) in the last step. 
Together with the fact that $\l\| \mb E H^2\l( X_{i_1}^{(1)},X_{i_2}^{(2)}\r) \r\| \leq \mb E\l\| \mb E_2 H^2\l( X_{i_1}^{(1)},X_{i_2}^{(2)}\r) \r\|$ for all $i_1,i_2$, and the inequality
\[
\l( \mb E\max_{i_1} \l\| \sum_{i_2:i_2\ne i_1} \mb E_2 H_{i_1,i_2}^2\l( X^{(1)}_{i_1}, X_{i_2}^{(2)}\r)\r\|^q \r)^{1/2q}
\leq
\l( \mb E\max_{i_1} \l\| \sum_{i_2:i_2\ne i_1} H_{i_1,i_2}^2 \l(X_{i_1}^{(1)},X_{i_2}^{(2)}\r) \r\|^q \r)^{1/(2q)}, 
\]
we obtain the following result.
\begin{corollary}
\label{cor:simple}
Under the assumptions of Theorem \ref{thm:degen-moment}, 
\begin{align*}
\l(\mb E \l\| U_n\r\|^{2q}  \r)^{1/2q}
\leq &
256/\sqrt{e}\Bigg[
r  \l( \sum_{i_1}\mb E\l\|  \sum_{i_2:i_2\ne i_1} \mb E_2 H_{i_1,i_2}^2\l( X^{(1)}_{i_1}, X_{i_2}^{(2)}\r)\r\| \r)^{1/2}
\\
&
+ 11 \, r^{3/2}\l( \mb E\max_{i_1} \l\| \sum_{i_2:i_2\ne i_1} H_{i_1,i_2}^2 \l(X_{i_1}^{(1)},X_{i_2}^{(2)}\r) \r\|^q \r)^{1/(2q)} 
\Bigg].
\end{align*}

\end{corollary}

\begin{remark}
Assume that $H_{i,j} = H$ is independent of $i,j$ and is such that $\|H(x_1,x_2)\|\leq M$ for all $x_1,x_2\in S$. 
Then
\[
\mb E\max_{i_1} \l\| \sum_{i_2:i_2\ne i_1} H_{i_1,i_2}^2\l(X_{i_1}^{(1)},X_{i_2}^{(2)}\r) \r\|^q 
\leq (n-1)^q M^{2q},
\]
and it immediately follows from Lemma \ref{moment-to-tail} and Corollary \ref{cor:simple} that for all $t\geq 1$ and an absolute constant $C>0$, 
\begin{align}
\label{eq:concentration}
&
\Pr\l( \l\| U_n \r\| \geq C\l( \sqrt{ \mb E\l\| \mb E_2 H^2(X_1^{(1)},X_2^{(2)})\r\| }\,\l(t+\log d\r)\cdot n + M\sqrt{n} \l(t+\log d\r)^{3/2}  \r) \r)\leq e^{-t}.
\end{align}
\end{remark}

\noindent Next, we obtain further refinements of the result that follow from estimating the term 
\[
r^{3/2}\l( \mb E\max_{i_1} \l\| \sum_{i_2:i_2\ne i_1} H_{i_1,i_2}^2 \l(X_{i_1}^{(1)},X_{i_2}^{(2)}\r) \r\|^q \r)^{1/(2q)} .
\]
\begin{lemma}
\label{lemma:remainder}
Under the assumptions of Theorem \ref{thm:degen-moment}, 
\begin{multline*}
r^{3/2} \l( \mb E\max_{i_1} \l\| \sum_{i_2:i_2\ne i_1} H_{i_1,i_2}^2 \l(X_{i_1}^{(1)},X_{i_2}^{(2)}\r) \r\|^q \r)^{1/(2q)} \\
\leq 
4e\sqrt{2} \sqrt{1 + \frac{\log d}{q}} \Bigg[  r \l( \sum_{i_1} \mb E \l\| \sum_{i_2:i_2\ne i_1} \mb E_2 H_{i_1,i_2}^2\l(X_{i_1}^{(1)},X_{i_2}^{(2)}\r) \r\|\r)^{1/2} 
\\
+ r^{3/2}\l(  \mb E\max_{i_1}\l\| \sum_{i_2:i_2\ne i_1} \mb E_2 H_{i_1,i_2}^2 \l(X_{i_1}^{(1)},X_{i_2}^{(2)}\r) \r\|^q \r)^{1/2q}
\\ 
+ r^2 \l( \sum_{i_1}\mb E\max_{i_2:i_2\ne i_1} \l\|  H_{i_1,i_2}^2 \l(X_{i_1}^{(1)},X_{i_2}^{(2)}\r) \r\|^q \r)^{1/2q}
\Bigg].
\end{multline*}
\end{lemma}
\begin{proof}
See Section \ref{proof:extraterm}.
\end{proof}

One of the key features of the bounds established above is the fact that they yield estimates for $\mb E \l\| U_n \r\|$: for example, 
Theorem \ref{thm:degen-moment} implies that 
\begin{align}
\label{eq:mom1}
\nonumber
\mb E\l\| U_n \r\|  \leq \,& C\log d\bigg( \l( \mb E \l\| \mb E_2 \wdt G \wdt G^\ast \r\| \r)^{1/2} + 
 \l\| \sum_{(i_1,i_2)\in I_n^2 } \mb E H^2_{i_1,i_2}\l(X_{i_1}^{(1)},X_{i_2}^{(2)}\r) \r\|^{1/2} \\
&
+\sqrt{\log d} \l( \mb E\max_{i_1} \l\| \sum_{i_2:i_2\ne i_1 } H^2_{i_1,i_2} \l(X_{i_1}^{(1)},X_{i_2}^{(2)}\r) \r\| \r)^{1/2}\Bigg)
\end{align}
for some absolute constant $C$. 
On the other hand, direct application of the non-commutative Khintchine's inequality \eqref{matrix-Khintchine} followed by Rosenthal's inequality (Lemma \ref{lemma:rosenthal-pd}) only gives that 
\begin{align}
\label{eq:mom3}
\nonumber
\mb E\l\| U_n \r\| \leq &C \log d \l( \sum_{i_1} \mb E\l\| \sum_{i_2:i_2\ne i_1} H_{i_1,i_2}^2\l(X_{i_1}^{(1)},X_{i_2}^{(2)}\r) \r\| \r)^{1/2} 
\\
\leq & C\log d \Bigg( \l(\sum_{i_1} \mb E \l\| \sum_{i_2:i_2\ne i_1} \mb E_2 H^2_{i_1,i_2}\l(X_{i_1}^{(1)},X_{i_2}^{(2)}\r) \r\| \r)^{1/2} 
\\ \nonumber
&+ \sqrt{\log d}\l( \sum_{i_1}\mb E \max_{i_2: i_2\ne i_1} \l\| H_{i_1,i_2}^2 \l(X_{i_1}^{(1)},X_{i_2}^{(2)}\r) \r\|\r)^{1/2} \Bigg),
%\leq C\frac{\log d}{n} \l( \mb E \l\| H^2\l(X_1^{(1)},X_2^{(2)}\r) \r\|\r)^{1/2}.
\end{align}
and it is easy to see that the right-hand side of \eqref{eq:mom1} is never worse than the bound \eqref{eq:mom3}. 
To verify that it can be strictly better, consider the framework of Example \ref{example:02}, where it is easy to check  (following the same calculations as those given in Section \ref{section:example-proof}) that
\begin{multline*}
 \l( \mb E \l\| \mb E_2 \wdt G \wdt G^\ast \r\| \r)^{1/2} = 1, \
 \l( \mb E\max_{i_1} \l\| \sum_{i_2:i_2\ne i_1 } H^2_{i_1,i_2} \l(X_{i_1}^{(1)},X_{i_2}^{(2)}\r) \r\| \r)^{1/2} = 1, 
 \\
 \l\| \sum_{(i_1,i_2)\in I_n^2 } \mb E H^2_{i_1,i_2}\l(X_{i_1}^{(1)},X_{i_2}^{(2)}\r) \r\|^{1/2} = 2 , 
\end{multline*}
while $\l(\sum_{i_1} \mb E \l\| \sum_{i_2:i_2\ne i_1} \mb E_2 H^2_{i_1,i_2}\l(X_{i_1}^{(1)},X_{i_2}^{(2)}\r) \r\| \r)^{1/2} = \sqrt{n}.$
\begin{remark}[Extensions to rectangular matrices]
All results in this section can be extended to the general case of $\mb C^{d_1\times d_2}$ - valued kernels by considering 
the Hermitian dilation $\m D(U_n)$ of $U_n$ as defined in \eqref{eq:dilation}, namely
\[
\m D(U_n) = \sum_{(i_1,i_2)\in I_n^2} \m D\l( H_{i_1,i_2}\l(X_{i_1}^{(1)},X_{i_2}^{(2)} \r) \r) \in \mb H^{d_1+d_2},
\]
and observing that $\|U_n\| = \l\| \m D(U_n)\r\|$. 
%we omit the general statements and only consider the special case of vector-valued U-statistics in Section \ref{section:vector}.
\end{remark}

%###############################
\section{Adamczak's moment inequality for U-statistics.}
\label{section:adamczak}
%###############################

The paper \cite{adamczak2006moment} by R. Adamczak developed moment inequalities for general Banach space-valued completely degenerate U-statistics of arbitrary order. 
More specifically, application of Theorem 1 in \cite{adamczak2006moment} to our scenario $\mathbb{B} = \l(\mathbb{H}^{d},~\|\cdot\|\r)$ and $m=2$ yields the following bounds for all $q\geq 1$ and $t\geq 2$:
\begin{align}
\label{adam-1}
&
\l( \mb E\| U_n \|^{2q} \r)^{1/(2q)} 
\leq C \Bigg( \mb E \l\| U_n \r\| 
%\sqrt{n(n-1)}\l( \mb E  \l\| H^2\l(X_{1}^{(1)},X_{2}^{(2)}\r) \r\| \r)^{1/2}
+ \sqrt{q}\cdot A + q \cdot B 
+ q^{3/2}\cdot \Gamma + q^2\cdot D \Bigg), \\
& 
\nonumber
\Pr\l( \|U_n\| \geq C \l( \mb E \l\| U_n \r\| +  \sqrt{t}\cdot A + t \cdot B 
+ t^{3/2}\cdot \Gamma + t^2\cdot D\r) \r)\leq e^{-t},
%q^{3/2}\expect{\max_{i_1} \l( \sum_{i_2} \mathbb{E}_2 \l\| \pi_{2,2} H \r\|^2 \r)^{q} }^{1/(2q)}  
\end{align} 
where $C$ is an absolute constant, and the quantities $A,B,\Gamma,D$ will be specified below (see Section \ref{section:adamczak-calculations} for the complete statement of Adamczak's result). 
Notice that inequality \eqref{adam-1} contains the ``sub-Gaussian'' term corresponding to $\sqrt{q}$ that did not appear in the previously established bounds. 

We should mention another important distinction between \eqref{adam-1} and the results of Theorem \ref{thm:degen-moment} and its corollaries, such as inequality \eqref{eq:concentration}: while \eqref{adam-1} describes the deviations of $\| U_n \|$ from its expectation, \eqref{eq:concentration} states that $U_n$ is close to its expectation \emph{as a random matrix}; similar connections exist between the Matrix Bernstein inequality \cite{tropp2012user} and Talagrand's concentration inequality \cite{boucheron2013concentration}. 
It particular, \eqref{adam-1} can be combined with a bound \eqref{eq:mom1} for $\mb E\| U_n \|$ to obtain a moment inequality that is superior (in a certain range of $q$) to the results derived from Theorem \ref{thm:degen-moment}. 
\begin{theorem}
\label{thm:adamczak}
Inequalities \eqref{adam-1} hold with the following choice of $A,B,\Gamma$ and $D$:
\begin{align*}
A = &
\sqrt{\log(de)} \l( \mb E\l\|  \mb E_2 \wdt G \wdt G^\ast \r\| +  
\l\|  \sum_{(i_1,i_2)\in I_n^2} \mb E H_{i_1,i_2}^2\l(X_{i_1}^{(1)},X_{i_2}^{(2)} \r)\r\| \r)^{1/2} \\
& +\log (de) \l( \mb E\max_{i_1} \l\| \sum_{i_2:i_2\ne i_1} H_{i_1,i_2}^2\l(X_{i_1}^{(1)},X_{i_2}^{(2)}\r) \r\| \r)^{1/2},
\\ 
B  = &  \l(  \sup_{z\in \mb C^d:\| z \|_2\leq 1}  \sum_{(i_1,i_2)\in I_n^2} \mb E\l( z^\ast H_{i_1,i_2}\l(X^{(1)}_{i_1},X^{(2)}_{i_2}\r) z  \r)^2\r)^{1/2} 
\\
&\leq \l(\l\| \sum_{(i_1,i_2)\in I_n^2} \mb E H_{i_1,i_2}^2(X^{(1)}_{i_1},X^{(2)}_{i_2})\r\| \r)^{1/2},  
\\
\Gamma = & \,\sqrt{1+\frac{\log d}{q}} \l( \sum_{i_1}\, \mb E_1 \l\| \sum_{i_2: i_2\ne i_1} \mb E_2  H_{i_1,i_2}^2\l(X^{(1)}_{i_1},X^{(2)}_{i_2}\r) \r\|^{q} \r)^{1/2q}, 
\\
D = & \l( \sum_{(i_1,i_2)\in I_n^2} \mb E\l\| H_{i_1,i_2}^2\l(X_{i_1}^{(1)},X_{i_2}^{(2)}\r) \r\|^{q} \r)^{1/(2q)}
\\
&+ \l( 1 + \frac{\log d}{q} \r)\l( \sum_{i_1}\,\mb E \max_{i_2: i_2\ne i_1} \l\| H^2_{i_1,i_2}\l(X^{(1)}_{i_1},X^{(2)}_{i_2}\r)\r\|^{q}\r)^{1/2q},
\end{align*}
where $\wdt G_i$ were defined in \eqref{eq:g}.
\end{theorem}
\begin{proof}
See Section \ref{section:adamczak-calculations}.
\end{proof}
% that the upper bound for $A$ coincides (up to absolute constants) with an upper bound for $\mb E\| U_n\|$ obtained in \eqref{eq:mom1}. 
\noindent It is possible to further simplify the bounds for $A$ (via Lemma \ref{lemma:remainder}) and $D$ to deduce that one can choose 
\begin{align}
\nonumber
A = &
\log(de) \l( \mb E\l\|   \mb E_2 \wdt G \wdt G^\ast \r\| +  
\l\|  \sum_{(i_1,i_2)\in I_n^2} \mb E H_{i_1,i_2}^2\l(X_{i_1}^{(1)},X_{i_2}^{(2)} \r)\r\| \r)^{1/2}, 
\\
\nonumber
B  = &   \l(  \sup_{z\in \mb C^d:\| z \|_2\leq 1}  \sum_{(i_1,i_2)\in I_n^2} \mb E\l( z^\ast H_{i_1,i_2}\l(X^{(1)}_{i_1},X^{(2)}_{i_2}\r) z  \r)^2\r)^{1/2}, 
\\
\nonumber
\Gamma = & \,(\log(de))^{3/2} \l( \sum_{i_1}\, \mb E_1 \l\| \sum_{i_2: i_2\ne i_1} \mb E_2  H_{i_1,i_2}^2\l(X^{(1)}_{i_1},X^{(2)}_{i_2}\r) \r\|^{q} \r)^{1/2q}, 
\\
\label{eq:adam-simplified}
D = & \log(de)\l( \sum_{(i_1,i_2)\in I_n^2} \mb E\l\| H_{i_1,i_2}^2\l(X_{i_1}^{(1)},X_{i_2}^{(2)}\r) \r\|^{q} \r)^{1/(2q)}.
\end{align}
The upper bound for $A$ can be modified even further as in \eqref{eq:d20}, using the fact that
\[
\mb E\l\| \mb E_2 \wdt G \wdt G^\ast \r\| \leq
\sum_{i_1} \mb E \l\| \sum_{i_2:i_2\ne i_1}\mb E_2 H_{i_1,i_2}^2\l(X^{(1)}_{i_1}, X_{i_2}^{(2)}\r) \r\|.
\]

\section{Proofs.}
\label{sec:proof}
%############

%################################
\subsection{Tools from probability theory and linear algebra.}
%################################

This section summarizes several facts that will be used in our proofs. 
The first inequality is a bound connecting the norm of a matrix to the norms of its blocks.
\begin{lemma}
\label{lemma:block-matrix}
Let $M\in \mb H^{d_1+ d_2}$ be nonnegative definite and such that 
$M=\begin{pmatrix}
A & X \\
X^\ast & B
\end{pmatrix}
$, where $A\in \mb H^{d_1}$ and $B\in \mb H^{d_2}$. 
Then 
\[
\vvvert M \vvvert\leq \vvvert A \vvvert + \vvvert B \vvvert
\]
for any unitarily invariant norm $\vvvert \cdot \vvvert$.
\end{lemma}
\begin{proof}
It follows from the result in \cite{bourin2012unitary} that under the assumptions of the lemma, there exist unitary operators $U,V$ such that 
\[
\begin{pmatrix}
A & X \\
X^\ast & B
\end{pmatrix} = 
U \begin{pmatrix}
A & 0 \\
0 & 0
\end{pmatrix} U^\ast +
V\begin{pmatrix}
0 & 0 \\
0 & B
\end{pmatrix} V^\ast,
\]
hence the result is a consequence of the triangle inequality. 
\end{proof}

The second result is the well-known decoupling inequality for U-statistics due to de la Pena and Montgomery-Smith \cite{PM-decoupling-1995}.
\begin{lemma}
\label{decoupling-lemma}
Let $\{X_i\}_{i=1}^n$ be a sequence of independent random variables with values in a measurable space $(S,\mathcal{S})$, and let $\{X_i^{(k)}\}_{i=1}^n$, $k=1,2,\ldots,m$ be $m$ independent copies of this sequence. 
Let $B$ be a separable Banach space and, for each $(i_1,\ldots,i_m)\in I^m_n$, let $H_{i_1,\ldots,i_m}: S^m\rightarrow B$ be a measurable function.  
%such that $\mb E\|H_{i_1,\ldots,i_m}(X_{i_1},\ldots, X_{i_m})\| < \infty$. 
Moreover, let $\Phi:[0,\infty)\rightarrow[0,\infty)$ be a convex nondecreasing function such that 
\[
\mb E\Phi(\|H_{i_1,\ldots,i_m}(X_{i_1},\ldots, X_{i_m})\|) < \infty 
\] 
for all $(i_1,\ldots,i_m)\in I^m_n$. Then
\begin{multline*}
\mb E\Phi\left(\left\|\sum_{(i_1,\ldots,i_m)\in I^m_n}H_{i_1,\ldots,i_m}(X_{i_1},\ldots, X_{i_m})\right\|\right) 
\leq 
\\
\mb E\Phi\left(C_m\left\|\sum_{(i_1,\ldots,i_m)\in I^m_n}H_{i_1,\ldots,i_m}\left(X_{i_1}^{(1)},\ldots, X_{i_m}^{(m)}\right)\right\|\right),
\end{multline*}
where $C_m:=2^m(m^m-1)\cdot((m-1)^{m-1}-1)\cdot\ldots\cdot 3$. 
Moreover, if $H_{i_1,\ldots,i_m}$ is $P$-canonical, then the constant $C_m$ can be taken to be $m^m$. 
Finally, there exists a constant $D_m>0$ such that for all $t>0$,
\begin{multline*}
\Pr\left(\left\|\sum_{(i_1,\ldots,i_m)\in I^m_n}H_{i_1,\ldots,i_m}(X_{i_1},\ldots, X_{i_m})\right\|\geq t\right)\\
\leq
D_m \Pr\left(D_m\left\|\sum_{(i_1,\ldots,i_m)\in I^m_n}H_{i_1,\ldots,i_m}\left(X_{i_1}^{(1)},\ldots, X_{i_m}^{(m)}\right)\right\|\geq t\right).
\end{multline*}
Furthermore, if $H_{i_1,\ldots,i_m}$ is permutation-symmetric, then, both of the above inequalities can be reversed (with different constants $C_m$ and $D_m$).
\end{lemma}

The following results are the variants of the non-commutative Khintchine's inequalities (that first appeared in the works by Lust-Piquard and Pisier) for the Rademacher sums and the Rademacher chaos with explicit constants, see \cite{lust1991non,lust1986inegalites}, page 111 in \cite{pisier1998non}, Theorems 6.14, 6.22 in \cite{HR-CS} and Corollary 20 in \cite{tropp2008conditioning}.
\begin{lemma}
\label{k-inequality}
Let $B_{j}\in\mathbb{C}^{r\times t} \ ,j=1,\ldots,n$ be the matrices of the same dimension, and let 
$\{\varepsilon_j\}_{j\in\mathbb{N}}$ be a sequence of i.i.d. Rademacher random variables. 
Then for any $p\geq1$,
\begin{align*}
&
\mb E\left\|\sum_{j=1}^n\varepsilon_j B_{j}\right\|_{S_{2p}}^{2p}
\leq
\l( \frac{2\sqrt 2}{e} p\r)^p \cdot
\max\left\{\left\|\left(\sum_{j=1}^n B_{j}B_{j}^*\right)^{1/2}\right\|_{S_{2p}}^{2p},
\left\|\left(\sum_{j=1}^n B_{j}^*B_{j}\right)^{1/2}\right\|_{S_{2p}}^{2p}\right\}.
\end{align*}
\end{lemma}

\begin{lemma}
\label{k-inequality-2}
Let $\{A_{i_1,i_2}\}_{i_1,i_2=1}^{n}$ be a sequence of Hermitian matrices of the same dimension, and let 
$\left\{\varepsilon_i^{(k)}\right\}_{i=1}^n, \ k=1,2,$ be i.i.d. Rademacher random variables. 
Then for any $p\geq 1$,
\begin{multline*}
\mb E\left\| \sum_{i_1=1}^n\sum_{i_2=1}^n A_{i_1,i_2}\varepsilon_{i_1}^{(1)}\varepsilon_{i_2}^{(2)} \right\|_{S_{2p}}^{2p}
\leq 
\\
2\l( \frac{2\sqrt 2}{e} p\r)^{2p}
\max\l\{  \l\| \l( G G^\ast \r)^{1/2}\r\|_{S_{2p}}^{2p}, \left\|\l(\sum_{i_1,i_2=1}^n A_{i_1,i_2}^2\r)^{1/2}\right\|_{S_{2p}}^{2p} \r\},
\end{multline*}
where the matrix $G\in\mathbb{H}^{nd}$ is defined as
\[
G:=
\left(
\begin{array}{cccc}
A_{11}  & A_{12}  & \ldots & A_{1n}  \\
A_{21}  & A_{22}  & \ldots & A_{2n}  \\
\vdots  & \vdots  &   \ddots & \vdots \\
A_{n1} & A_{n2} & \ldots & A_{nn} 
\end{array}
\right).
\]
\end{lemma}

The following result (Theorem A.1 in \cite{chen2012masked}) is a variant of matrix Rosenthal's inequality for nonnegative definite matrices.  
\begin{lemma}
\label{lemma:rosenthal-pd}
Let $Y_1,\ldots,Y_n\in \mb H^d$ be a sequence of independent nonnegative definite random matrices. 
Then for all $q\geq 1$ and $r=\max(q, \log(d))$,
\[
\l( \mb E \l\| \sum_j Y_j \r\|^q \r)^{1/2q} \leq \l\| \sum_j \mb E Y_j \r\|^{1/2} + 2\sqrt{2er} \l( \mb E\max_j \|Y_j\|^q\r)^{1/2q}. 
\] 
\end{lemma}

The next inequality (see equation (2.6) in \cite{gine2000exponential}) allows to replace the sum of moments of nonnegative random variables with maxima. 
\begin{lemma}
\label{lemma:sum-max}
Let $\xi_1,\ldots,\xi_n$ be independent random variables. Then for all $q>1$ and $\alpha\geq 0$, 
\[
q^{\alpha q} \sum_{i=1}^n |\xi_i|^q \leq 2(1 + q^{\alpha})\max\l( q^{\alpha q} \mb E\max_i |\xi_i|^q, \l(\sum_{i=1}^n \mb E|\xi_i| \r)^q \r).
\] 
\end{lemma}

Finally, the following inequalities allow transitioning between moment and tail bounds. 
\begin{lemma}
\label{moment-to-tail}
Let $X$ be a random variable satisfying
$
\l(\mb E |X|^p\r)^{1/p}\leq a_4 p^2 + a_3 p^{3/2} + a_2 p + a_1\sqrt{p} + a_0
$
for all $p\geq 2$ and some positive real numbers $a_j, \ j=0,\ldots,3$. 
Then for any $u\geq 2$,
\[
\Pr \left(|X| \geq e(a_4 u^2 + a_3 u^{3/2} + a_2 u + a_1\sqrt{u} + a_0)\right)\leq\exp\left(-u\right).
\]
\end{lemma}
\noindent See Proposition 7.11 and 7.15 in \cite{foucart2013mathematical} for the proofs of closely related bounds. 
\begin{lemma}
\label{tail-to-moment}
Let $X$ be a random variable such that 
$
\Pr\l( |X| \geq a_0 + a_1\sqrt{u} + a_2 u \r) \leq e^{-u}
$
for all $u\geq 1$ and some $0\leq a_0,a_1,a_2<\infty$. Then 
\[
\l( \mb E |X|^p \r)^{1/p}\leq C(a_0 + a_1\sqrt{p} + a_2 p)
\]
for an absolute constant $C>0$ and all $p\geq 1$.
%Let $X$ be a random variable such that 
%$\Pr\l( |X| \geq a_1u + a_2\sqrt{u} \r) \leq Me^{-u}$
%for some $M\geq 2$,	 all $u\geq \log(M)$ and some $0\leq a_1,a_2<\infty$. Then 
%\[
%\l( \mb E |X|^p \r)^{1/p}\leq C\l( a_1 (\log(2M) + p) + a_2 (\sqrt{\log(2M)}+\sqrt{p})\r)
%\]
%for some absolute constant $C>0$ all $p\geq 1$.
\end{lemma}
The proof follows from the formula $\mb E|X|^p = p\int_0^\infty \Pr\left( |X|\geq t\right)t^{p-1} dt$, see Lemma A.2 in \cite{dirksen2015tail} and Proposition 7.14 in \cite{foucart2013mathematical} for the derivation of similar inequalities. 
Next, we will use Lemma \ref{decoupling-lemma} combined with a well-known argument to obtain the symmetrization inequality for degenerate U-statistics.
\begin{lemma}
\label{lemma:symmetry}
Let $H_{i_1,i_2}:S\times S\mapsto \mb H^d$ be degenerate kernels, $X_1,\ldots,X_n$ -- i.i.d. $S$-valued random variables, and assume that $\{X_i^{(k)}\}_{i=1}^n$, $k=1,2,$ are independent copies of this sequence. 
Moreover, let $\left\{\varepsilon_i^{(k)}\right\}_{i=1}^n, \ k=1,2,$ be i.i.d. Rademacher random variables. 
Define
%\frac{(n-2)!}{n!}
\begin{align}
\label{eq:v'}
U'_n := \sum_{(i_1,i_2)\in I^2_n}\varepsilon_{i_1}^{(1)}\varepsilon_{i_2}^{(2)} H_{i_1,i_2}\l(X_{i_1}^{(1)},X_{i_2}^{(2)}\r).
\end{align}
Then for any $p\geq1$,
\begin{equation*}
\Big( \mb E\l\| U_n  \r\|^p \Big)^{1/p}
\leq 16\Big( \mb E\l\| U'_n  \r\|^p \Big)^{1/p}.
\end{equation*}
\end{lemma}

%########################################
%\subsection{Proof of Lemma \ref{lemma:symmetry}}
%\label{proof:symmetry}
%########################################
\begin{proof}
Note that 
\begin{align*}
\mb E\|U_n\|^p =& \mb E \l\|\sum_{(i_1,i_2)\in I^2_n} H_{i_1,i_2}(X_{i_1},X_{i_2})\r\|^p \\
\leq
&  \mb E\l\| 2^2\sum_{(i_1,i_2)\in I^2_n}  H_{i_1,i_2}\l(X_{i_1}^{(1)},X_{i_2}^{(2)}\r)\r\|^p,
\end{align*}
where the inequality follows from the fact that $H_{i_1,i_2}$ is $\mathcal{P}$-canonical, hence Lemma \ref{decoupling-lemma} applies with constant equal to $C_2=4$. 

Next, for $i=1,2$, let $\mathbb{E}_i[\cdot]$ stand for the expectation with respect to $\l\{ X_j^{(i)}, \eps_j^{(i)} \r\}_{j\geq 1}$ only (that is, conditionally on $\l\{ X_j^{(k)}, \eps_j^{(k)} \r\}_{j\geq 1}, \ k\ne i$).
Using iterative expectations and the symmetrization inequality for the Rademacher sums twice (see Lemma 6.3 in \cite{LT-book-1991}), we deduce that 
\begin{align*}
\mb E\|U_n\|^p
\leq
&4^p\,\mb E\left\| \sum_{(i_1,i_2)\in I^2_n} H_{i_1,i_2}\left(X_{i_1}^{(1)},X_{i_2}^{(2)}\right)\right\|^p
\\
=&4^p\,\mb E\l[ \mathbb{E}_1 \left\|  \sum_{i_1=1}^n\sum_{i_2\ne i_1} H_{i_1,i_2}\left(X_{i_1}^{(1)},X_{i_2}^{(2)}\right)\right\|^p\right] 
\\ \leq &
4^p\,\expect{\mathbb{E}_1 \left\| 2\sum_{i_1=1}^n \varepsilon_{i_1}^{(1)} \sum_{i_2\ne i_1} 
H_{i_1,i_2}\left(X_{i_1}^{(1)},X_{i_2}^{(2)}\right)\right\|^p }
\\ =&
4^p\,\expect{\mathbb{E}_2 \left\| 2\sum_{i_2=1}^n 
\sum_{i_1\ne i_2} \varepsilon_{i_1}^{(1)} H_{i_1,i_2}\left(X_{i_1}^{(1)},X_{i_2}^{(2)}\right)\right\|^p }
\\ \leq &
4^p\,\expect{\left\| 4 \sum_{(i_1,i_2)\in I^2_n}\varepsilon_{i_1}^{(1)}\varepsilon_{i_2}^{(2)}
H_{i_1,i_2}\left(X_{i_1}^{(1)},X_{i_2}^{(2)}\right)\right\|^p}.
\end{align*}
\end{proof}

%############################################
\subsection{Proofs of results in Section \ref{sec:moment}.}
%#############################################
\subsubsection{Proof of Lemma \ref{bound-on-expectation}.}
%#############################################

Recall that 
\[
X=\sum_{i_1=1}^n\sum_{i_2\neq i_1} A_{i_1,i_2}\varepsilon_{i_1}^{(1)}\varepsilon_{i_2}^{(2)},
\] 
where $A_{i_1,i_2}\in \mb H^d$ for all $i_1,i_2$, and let 
$C_p:=2\l( \frac{2\sqrt 2}{e} p\r)^{2p}$. 
We will first establish the upper bound. 
Application of Lemma \ref{k-inequality-2} (Khintchine's inequality) to the sequence of matrices $\l\{ A_{i_1,i_2}\r\}_{i_1,i_2=1}^n$ such that $A_{j,j}=0$ for $j=1,\ldots,n$ yields
\begin{align}
\label{main-bound}
&
\l(\mb E\|X\|_{S_{2p}}^{2p} \r)^{1/2p} \leq 2^{1/2p} \frac{2\sqrt 2 }{e}\cdot p \cdot
\max\l\{\l\|\l(GG^\ast\r)^{1/2}\r\|_{S_{2p}}, \l\|\l(\sum_{(i_1,i_2)\in I_n^2} A_{i_1,i_2}^2\r)^{1/2}\r\|_{S_{2p}}\r\},
%\max\l\{  \l\| G G^T\r\|, \left\|\sum_{i_1,i_2=1}^n A_{i_1,i_2}^2\right\| \r\}^{p},
\end{align}
where 
\[
G:=
\left(
\begin{array}{cccc}
0  & A_{12}  & \ldots & A_{1n}  \\
A_{21}  & 0  & \ldots & A_{2n}  \\
\vdots  & \vdots  &   \ddots & \vdots \\
A_{n1} & A_{n2} & \ldots & 0 
\end{array}
\right) \in \mb R^{nd\times nd}.
\]
Our goal is to obtain a version of inequality \eqref{main-bound} for $p=\infty$. 
To this end, we need to find an upper bound for 
\[
\inf_{p\geq q} \l[ \frac{p \cdot
\max\l\{\l\|\l(GG^\ast\r)^{1/2}\r\|_{S_{2p}}, \l\|\l(\sum_{(i_1,i_2)\in I_n^2} A_{i_1,i_2}^2\r)^{1/2}\r\|_{S_{2p}}\r\}}
{\max\l\{\l\|\l(GG^\ast\r)^{1/2}\r\|, \l\|\l(\sum_{(i_1,i_2)\in I_n^2} A_{i_1,i_2}^2\r)^{1/2}\r\| \r\}} \r].
\]
Since $G$ is a $nd\times nd$ matrix, a naive upper bound is of order $\log(nd)$. 
We will show that it can be improved to $\log d$.
To this end, we need to distinguish between the cases when the maximum in \eqref{main-bound} is attained by the first or second term. 
Define 
\[
\widehat{B}_{i_1,i_2} = [0~|~0~|\ldots|~A_{i_1,i_2}~|\ldots|~0~|~0]\in\mathbb{C}^{d\times nd},
\]
where $A_{i_1i_2}$ sits on the $i_1$-th position of the above block matrix. Moreover, let 
\begin{align}
\label{eq:b}
B_{i_2}=\sum_{i_1:i_1\neq i_2}\widehat{B}_{i_1,i_2}.
\end{align}
Then it is easy to see that 
\begin{align*}
GG^\ast & = \sum_{i_2} B_{i_2}^\ast  B_{i_2}, \\
\sum_{(i_1,i_2)\in I_n^2} A_{i_1,i_2}^2 & = \sum_{i_2}  B_{i_2}  B_{i_2}^\ast.
\end{align*}
The following bound gives a key estimate.

\begin{lemma}
\label{lemma:eigen-compare}
Let $M_1,\ldots,M_N$ be a sequence of $\mb C^{d\times nd}$-valued matrices. 
Let $\lambda_1,\ldots,\lambda_{nd}$ be eigenvalues of $\sum_{j} M_{j}^\ast M_{j}$ and let 
$\nu_1,\ldots,\nu_d$ be eigenvalues of $\sum_{j} M_{j}  M_{j}^\ast$. 
Then 
$\sum_{i=1}^{nd}\lambda_i = \sum_{j=1}^d\nu_j$. 
Furthermore,
if $\max_i\lambda_i\leq\frac{1}{d}\sum_{j=1}^d\nu_j$, then
\[
\l\|\l(\sum_{j}M_{j} M_{j}^\ast\r)^{1/2}\r\|_{S_{2p}}^{2p} \geq \l\|\l(\sum_{j} M_{j}^\ast M_{j}\r)^{1/2}
\r\|_{S_{2p}}^{2p},
\]
for any integer $p\geq2$.
\end{lemma}
The proof of the Lemma is given in Section \ref{sec:proof-eigencompare}. 
We will apply this fact with $M_j = B_j$, $j=1,\ldots,n$.  
Assuming that $\max_i \lambda_i\leq\frac{1}{d}\sum_{j=1}^{nd}\lambda_j$, it is easy to see that the second term in the maximum in \eqref{main-bound} dominates, hence
\begin{multline}
\mb E\|X\|_{S_{2p}}^{2p} \leq C_p \left\|\left(\sum_{(i_1,i_2)\in I_n^2} A_{i_1,i_2}^2\right)^{1/2}\right\|_{S_{2p}}^{2p}
=C_p \, \tr \left(\sum_{(i_1,i_2)\in I_n^2}A_{i_1i_2}^2\right)^p
\\
\leq C_p \cdot d \cdot\left\|\left(\sum_{(i_1,i_2)\in I_n^2} A_{i_1i_2}^2\right)^p\right\| 
= C_p \cdot d\cdot\left\|\sum_{(i_1,i_2)\in I^2_n}A_{i_1,i_2}^2\right\|^p,
\label{inter-square-bound}
\end{multline}
where the last equality follows from the fact that for any positive semidefinite matrix $H$, $\|H^p\|=\|H\|^p$. 
On the other hand, when $\max_i \lambda_i>\frac{1}{d}\sum_{j=1}^{nd}\lambda_j$, it is easy to see that for all $p\geq 1$,
\begin{align*}
d>\sum_{j=1}^{nd}\frac{\lambda_j}{\max_i \lambda_i}\geq
\sum_{j=1}^{nd}\l(\frac{\lambda_j}{\max_i \lambda_i}\r)^{p},
\end{align*}
which in turn implies that
\begin{align}
\label{eq:eig1}
d \l(\max_i\lambda_i\r)^{p}\geq\sum_{j=1}^{nd}\lambda_j^{p}.
\end{align}
Moreover, 
\begin{align}
\label{eq:eig2}
\left\|\left(\sum_{i_2}B_{i_2}^\ast B_{i_2}\right)^{1/2}
\right\|_{S_{2p}}^{2p} = 
\tr\l(\left(\sum_{i_2}B_{i_2}^\ast B_{i_2}\right)^p\r) = \sum_{i=1}^{nd}\lambda_i^{p}.
\end{align}
Combining \eqref{eq:eig1}, \eqref{eq:eig2}, we deduce that
\[
\left\|\left(\sum_{i_2}B_{i_2}^\ast B_{i_2} \right)^{1/2}
\right\|_{S_{2p}}^{2p} \leq d \l\| \l(\sum_{i_2}B_{i_2}^\ast B_{i_2}\r)^p \r\| = d\l\| \l(GG^\ast\r)^p \r\|
=d\l\| GG^\ast \r\|^p,
\]
where the second from the last equality follows again from the fact that for any positive semi-definite matrix $H$, $\|H^p\|=\|H\|^p$. Thus, combining the bound above with \eqref{main-bound} and \eqref{inter-square-bound}, we obtain
\[
\mb E\|X\|_{S_{2p}}^{2p}
\leq d\cdot C_p\max\l\{  \l\| GG^\ast \r\|^p, \left\|\sum_{(i_1,i_2)\in I^2_n}A_{i_1,i_2}^2\right\|^p \r\}.
\]
Finally, set $p=\max(q, \log(d))$ and note that $d^{1/2p}\leq \sqrt{e}$, hence
\[
\l( \mb E\|X\|^{2q}\r)^{1/2q}\leq \l( \mb E\|X\|^{2p}\r)^{1/2p}\leq \frac{4}{\sqrt{e}}\max\{\log d, q\}\cdot\max\l\{  \l\| GG^\ast\r\|, \left\|\sum_{(i_1,i_2)\in I^n_2}A_{i_1,i_2}^2\right\| \r\}^{1/2}.
\]
This finishes the proof of upper bound.

Now, we turn to the lower bound. 
Let $\mb E_{1}[\cdot]$ stand for the expectation with respect to $\l\{ \eps_j^{(1)}\r\}_{j\geq 1}$ only. 
Then 
%By Jensen's inequality applied to $f(x)=x^{1/2}$, for any $p\geq1$,
\begin{align*}
\l( \mb E\|X\|^{2p}\r)^{1/(2p)}\geq& \l(\mb E\|X\|^{2}\r)^{1/2} =
\l(\mathbb{E}\mathbb{E}_1 \left\| \l(\sum_{(i_1,i_2)\in I_n^2} \varepsilon_{i_1}^{(1)}\varepsilon_{i_2}^{(2)}A_{i_1,i_2} \r)^2\right\| \r)^{1/2}\\
\geq&
\l( \mathbb{E}\left\|\mathbb{E}_1\l(\sum_{(i_1,i_2)\in I_n^2} \varepsilon_{i_1}^{(1)}\varepsilon_{i_2}^{(2)}A_{i_1,i_2}\r)^2 \right\| \r)^{1/2}\\
=& 
\l( \mathbb{E}\left\|\sum_{i_1} \l(\sum_{i_2:i_2\neq i_1}\varepsilon_{i_2}^{(2)}A_{i_1,i_2}\r)^2\right\| \r)^{1/2}.
\end{align*}
It is easy to check that 
\[
\sum_{i_1=1}^n\left(\sum_{i_2:i_2\neq i_1}\varepsilon_{i_2}^{(2)}A_{i_1,i_2}\right)^2 = 
\left(\sum_{i_2}\varepsilon_{i_2}^{(2)}B_{i_2}\right) \left(\sum_{i_2}\varepsilon_{i_2}^{(2)} B_{i_2}\right)^\ast,
\]
where $B_i$ were defined in \eqref{eq:b}. Hence
\[
\l(\mb E\|X\|^{2p}\r)^{1/(2p)}\geq \l( \mb E\l\|\left(\sum_{i_2}\varepsilon_{i_2}^{(2)}B_{i_2}\right)
\left(\sum_{i_2}\varepsilon_{i_2}^{(2)}B_{i_2}\right)^\ast \r\| \r)^{1/2}.
\]
Next, for any matrix $A\in\mathbb{C}^{d_1\times d_2}$, 
\begin{align*}
\l\|
\left(
\begin{array}{ccc}
0  &  A^\ast     \\
A  &   0   \\   
\end{array}
\right)^2
\r\|
=
\l\|
\left(
\begin{array}{ccc}
A^\ast A  &  0    \\
0  &   A A^\ast   \\   
\end{array}
\right)
\r\| 
= \max\{\|A^\ast A\|, \|A A^\ast\|\}=\|AA^\ast\|,
\end{align*}
where the last equality follows from the fact that $\|AA^\ast\|=\|A^\ast A\|$. 
Taking $A=\sum_{i_2}\varepsilon_{i_2}^{(2)}B_{i_2}$ yields that
\begin{align*}
\l(\mb E\|X\|^{2p}\r)^{1/(2p)}
\geq& 
\l( \mb E \l\| BB^\ast \r\| \r)^{1/2}
=\l( \mb E
  \l\| \l( \sum_{i_2}\varepsilon_{i_2}^{(2)}\left(
\begin{array}{ccc}
0  &  B_{i_2}^\ast     \\
B_{i_2}  &   0   \\   
\end{array}
\right)  \r)^2\r\| \r)^{1/2}
\\
\geq &
\l\| \mb E \l( \sum_{i_2}\varepsilon_{i_2}^{(2)}\left(
\begin{array}{ccc}
0  &  B_{i_2}^\ast     \\
B_{i_2}  &   0   \\   
\end{array}
\right) \r)^2 \r\|
^{1/2}
=\l\|\sum_{i_2}\left(
\begin{array}{ccc}
B_{i_2}^\ast B_{i_2}  & 0      \\
0  &  B_{i_2} B_{i_2}^\ast  \\   
\end{array}
\right)  \r\|
^{1/2}\\
=&\max\l\{\l\| \sum_{i_2}B_{i_2}^\ast B_{i_2} \r\|,  \l\| \sum_{i_2}B_{i_2} B_{i_2}^\ast \r\|\r\}^{1/2}
= \max\l\{  \l\| GG^\ast\r\|, \left\|\sum_{(i_1,i_2)\in I^2_n}A_{i_1,i_2}^2\right\| \r\}^{1/2}.
\end{align*}
%This finishes the proof of the lower bound.

\subsubsection{Proof of Lemma \ref{lemma:eigen-compare}.}
\label{sec:proof-eigencompare}
%#########################################

The equality of traces is obvious since
\[
\tr\l( \sum_{j=1}^N M_j M_j^\ast\r) = \sum_{j=1}^N  \tr \l(M_j M_j ^\ast\r) = 
\sum_{j=1}^N  \tr \l(M_j^\ast M_j \r) = \tr\l( \sum_{j=1}^N M_j^\ast M_j\r).
\]
Set
\[
S:=\sum_{i=1}^{nd}\lambda_i = \sum_{i=1}^d\nu_i.
\]
Note that 
\begin{align*}
&\l\|\l(\sum_{j=1}^N M_{j}^\ast M_{j}\r)^{1/2}\r\|_{S_{2p}}^{2p} = \tr\l(\l(\sum_{j=1}^N M_{j}^\ast M_{j}\r)^p\r) = 
\sum_{i=1}^{nd}\lambda_i^p, \\
&\l\|\l(\sum_{j=1}^N M_{j} M_{j}^\ast\r)^{1/2}\r\|_{S_{2p}}^{2p}=\tr\l(\l(\sum_{j=1}^N M_{j}M_{j}^\ast\r)^p\r) = 
\sum_{i=1}^{d}\nu_i^p.
\end{align*}
Moreover, $\lambda_i\geq0$, $\nu_j\geq0$ for all $i,j$, and $\max_i\lambda_i\leq \frac{1}{d}\sum_{j=1}^d \nu_j = \frac{S}{d}$ by assumption. 
It is clear that 
\begin{align*}
&\l\|\l(\sum_{j=1}^N M_{j}^\ast M_{j}\r)^{1/2}\r\|_{S_{2p}}^{2p} \leq \max_{0\leq\lambda_i\leq\frac{S}{2d},~\sum_{i=1}^{nd}\lambda_i= S}\sum_{i=1}^{nd}\lambda_i^p, \\
&\l\|\l(\sum_{j=1}^N M_{j} M_{j}^\ast\r)^{1/2}\r\|_{S_{2p}}^{2p}\geq\min_{\nu_i\geq0~,\sum_{i=1}^{d}\nu_i= S}\sum_{i=1}^{d}\nu_i^p.
\end{align*}
Hence, it is enough to show that
\begin{align}
\label{optimize-problem}
\max_{0\leq\lambda_i\leq\frac{S}{d},~\sum_{i=1}^{nd}\lambda_i=S}\sum_{i=1}^{nd}\lambda_i^p
\leq \min_{\nu_i\geq0,~\sum_{i=1}^{d}\nu_i= S}\sum_{i=1}^{d}\nu_i^p.
\end{align}
The right hand side of the inequality \eqref{optimize-problem} can be estimated via Jensen's inequality as
\begin{multline}
\label{rhs}
\min_{\nu_i\geq0,~\sum_{i=1}^{d}\nu_i= S}\sum_{i=1}^{d}\nu_i^p
=d\cdot\min_{\nu_i\geq0,~\sum_{i=1}^{d}\nu_i= S}\frac1d\sum_{i=1}^{d}\nu_i^p \\
\geq d\cdot\min_{\nu_i\geq0,~\sum_{i=1}^{d}\nu_i= S}\l(\frac{1}{d}\sum_{i=1}^{d}\nu_i\r)^p
= d\cdot\l(\frac{S}{d}\r)^{p}.
%=\frac{S^p}{d^{p-1}}.
\end{multline}
It remains to show that $\sum_{i=1}^{nd} \lambda_i^p \leq d\cdot\l(\frac{S}{d}\r)^{p}$. 
For a sequence $\l\{ a_j \r\}_{j=1}^N \subset \mb R$, let $a_{(j)}$ be the j-th smallest element of the sequence, where the ties are broken arbitrary. 
A sequence $\l\{ a_j \r\}_{j=1}^N$ majorizes a sequence $\l\{ b_j \r\}_{j=1}^N$ whenever 
$\sum_{j=0}^k a_{(N-j)} \geq \sum_{j=0}^k b_{(N-j)}$ for all $0\leq k\leq N-2$, and $\sum_j a_j = \sum_j b_j$. 
A function $g:\mb R^N\mapsto \mb R$ is called Schur-convex if $g( a_1,\ldots, a_N)\geq g(b_1,\ldots,b_N)$ whenever 
$\l\{ a_j \r\}_{j=1}^N$ majorizes $\l\{ b_j \r\}_{j=1}^N$. It is well known that if $f:\mb R\mapsto \mb R$ is convex, then $g(a_1,\ldots,a_N)=\sum_{j=1}^N f(a_j)$ is Schur convex. In particular, 
$g(a_1,\ldots,a_N) = \sum_{j=1}^N a_j^p$, where $a_1,\ldots,a_N\geq 0$, is Schur convex for $p\geq 1$. 
Consider the sequence $a_1=\ldots=a_d = \frac{S}{d}, \ a_{d+1}=\ldots = a_{nd}=0$ and $b_1 = \lambda_1,\ldots,b_{nd} = \lambda_{nd}$. 
Since $\max_i\lambda_i\leq\frac{S}{d}$ by assumption, the sequence $\{a_j\}$ majorizes $\{b_j\}$, hence Schur convexity yields that $\sum_{i=1}^{nd} \lambda_i^p \leq \sum_{i=1}^d \l( \frac{S}{d}\r)^p = d\cdot \l( \frac{S}{d}\r)^p$, implying the result. 
\footnote{We are thankful to the anonymous Referee for suggesting an argument based on Schur convexity, instead of the original proof that was longer and not as elegant.}

\subsubsection{Proof of Theorem \ref{thm:degen-moment}.}
\label{proof:degen-moment}
%##############################################

The first inequality in the statement of the theorem follows immediately from Lemma \ref{decoupling-lemma}. Next, it is easy to deduce from the proof of Lemma \ref{lemma:symmetry} that
\begin{equation}
\label{initial-1}
%\Big( \mb E\l\| U_n  \r\|^{2q} \Big)^{1/2q}
 \l(\mb E \l\| \sum_{(i_1,i_2)\in I^2_n} H_{i_1,i_2}\l(X_{i_1}^{(1)},X_{i_2}^{(2)}\r) \r\|^{2q}\r)^{1/2q} 
\leq 4\Big( \mb E\l\| U'_n  \r\|^{2q} \Big)^{1/2q},
\end{equation}
where $U'_n$ was defined in \eqref{eq:v'}. 
Applying Lemma \ref{bound-on-expectation} conditionally on $\{X_i^{(j)}\}_{i=1}^{n}, \ j=1,2$, we get
\begin{multline}
\label{inter-chaos-1}
\Big( \mb E\l\| U'_n  \r\|^{2q} \Big)^{1/2q}=
\l( \mb E\l\| \sum_{(i_1,i_2)\in I_n^2}\varepsilon_{i_1}^{(1)}\varepsilon_{i_2}^{(2)} H_{i_1,i_2}(X_{i_1}^{(1)},X_{i_2}^{(2)}) \r\|^{2q} \r)^{1/(2q)} 
\\
\leq 4e^{-1/2}\max(q,\log d) \,
\l( \mb E\max\l\{ \| \wdt G \wdt G^\ast\|, \l\|\sum_{(i_1,i_2)\in I_n^2} H_{i_1,i_2}^2\l(X^{(1)}_{i_1},X^{(2)}_{i_2}\r) \r\|   \r\}^q  \r)^{1/2q},
\end{multline}
where $\wdt G$ was defined in \eqref{eq:g}.  
%\[A_{i_1,i_2} = H\l( X_{i_1}^{(1)}, X_{i_2}^{(2)}\r),~i_1\neq i_2.\]
Let $\wdt G_i$ be the $i$-th column of $\wdt G$, then
\[
\wdt G \wdt G^\ast = \sum_{i=1}^n \wdt G_i \wdt G_i^\ast,
~~\sum_{(i_1,i_2)\in I_n^2} H_{i_1,i_2}^2(X_{i_1}^{(1)},X_{i_2}^{(2)}) = \sum_{i=1}^n \wdt G_i^\ast \wdt G_i.
\]
Let $Q_i\in\mathbb{H}^{(n+1)d\times(n+1)d}$ be defined as
\[
Q_i = 
\left(
\begin{array}{ccc}
0  &  \wdt G_{i}^\ast     \\
\wdt G_{i}  &   0   \\   
\end{array}
\right),
\]
so that
\[
Q_i^2 = 
\left(
\begin{array}{ccc}
\wdt G_{i}^\ast \wdt G_i & 0     \\
0 & \wdt G_{i} \wdt G_i^\ast   \\   
\end{array}
\right).
\]
Inequality \eqref{inter-chaos-1} implies that
\begin{equation}
\label{inter-chaos-2}
\l( \mb E\l\| \sum_{(i_1,i_2)\in I_n^2}\varepsilon_{i_1}^{(1)}\varepsilon_{i_2}^{(2)} H_{i_1,i_2}\l(X_{i_1}^{(1)},X_{i_2}^{(2)}\r) \r\|^{2q} \r)^{1/(2q)}
\leq 4e^{-1/2}\max(q,\log d) \l( \mb E\l\| \sum_{i=1}^n Q_i^2 \r\|^q\r)^{1/(2q)}.
\end{equation}
Let $\mathbb{E}_2[\cdot]$ stand for the expectation with respect to $\l\{X^{(2)}_i\r\}_{i=1}^n$ only (that is, conditionally on 
$\l\{X^{(1)}_i\r\}_{i=1}^n$). 
Then Minkowski inequality followed by the symmetrization inequality imply that
\begin{align}
\l( \mb E\l\| \sum_{i=1}^n Q_i^2 \r\|^q\r)^{1/(2q)} \leq
& 
\l( \mb E\l\| \sum_{i=1}^n \l( Q_i^2 - \mb E_2 Q_i^2\r)  \r\|^q \r)^{1/(2q)}  + 
\l( \mb E \l\| \sum_{i=1}^n \mb E_2 Q_i^2 \r\|^{q}  \r)^{1/2q}
\nonumber\\
=& \l( \mathbb{E}\,\mathbb{E}_2 \l\|  \sum_{i=1}^n Q_i^2 - \mb E_2{Q_i^2} \r\|^q \r)^{1/(2q)} + 
\l( \mb E \l\| \sum_{i=1}^n \mb E_2 Q_i^2 \r\|^{q}  \r)^{1/2q}  
\nonumber\\
 \leq& \sqrt 2 \l( \mb E\l\|  \sum_{i=1}^n\varepsilon_i Q_i^2 \r\|^q \r)^{1/(2q)} + 
\l( \mb E \l\| \sum_{i=1}^n \mb E_2 Q_i^2 \r\|^{q}  \r)^{1/2q}.
 \label{inter-chaos-3}
\end{align}
%\lcomm{Not clear how to get a "good" bound on $\l( \mb E \l\| \sum_{i=1}^n \mb E_2 Q_i^2 \r\|^{q}  \r)^{1/2q}$... }
\noindent 
Next, we obtain an upper bound for $ \l( \mb E\l\|  \sum_{i=1}^n\varepsilon_i Q_i^2 \r\|^q \r)^{1/(2q)}$. 
To this end, we apply Khintchine's inequality (Lemma \ref{k-inequality}). 
Denote $C_r:=\l( \frac{2\sqrt 2}{e} r\r)^{2r}$, and let 
$\mathbb{E}_{\varepsilon}[\cdot]$ be the expectation with respect to $\{\varepsilon_i\}_{i=1}^n$ only. 
Then for $r>q$ we deduce that
\begin{align*}
\mathbb{E}_{\varepsilon}\l\| \sum_{i=1}^n\varepsilon_i Q_i^2 \r\|_{S_{2r}}^{2r} \leq& 
C_r^{1/2} \l\| \l(\sum_{i=1}^n Q_i^4\r)^{1/2}  \r\|_{S_{2r}}^{2r} 
\\
=& C_r^{1/2} \tr \l(\sum_{i=1}^n Q_i^4\r)^r \leq C_r^{1/2} \tr \l(\sum_{i=1}^n Q_i^2\cdot \l\| Q_i^2 \r\|\r)^r \\
\leq& C_r^{1/2} \max_i \l\| Q_i^2 \r\|^r  \cdot  \l\| \l(\sum_{i=1}^n Q_i^2\r)^{1/2}  \r\|_{S_{2r}}^{2r},
\end{align*}
where we used the fact that $Q_i^4\preceq \|Q_i^2\| Q_i^2$ for all $i$, and the fact that $A\preceq B$ implies that $\tr \,g(A)\leq \tr \,g(B)$ for any non-decreasing $g:\mb R\mapsto \mb R$.
% as well as monotonicity of the trace. 
Next, we will focus on the term 
\[
\l\| \l( \sum_{i=1}^n Q_i^2 \r)^{1/2}  \r\|_{S_{2r}}^{2r}
= \tr\l( \l(\sum_{i=1}^n \wdt G_i \wdt G_i^\ast\r)^r \r) + \tr\l( \l(\sum_{i=1}^n \wdt G_i^\ast \wdt G_i\r)^r \r).
\] 
Applying Lemma \ref{lemma:eigen-compare} with $M_j = \wdt G_j^\ast, \ j=1,\ldots,n$, we deduce that
\begin{itemize}
\item if $\l\| \sum_{i=1}^n \wdt G_i \wdt G_i^\ast  \r\| \leq \frac{1}{d}\tr\l( \sum_{i=1}^n \wdt G_i^\ast \wdt G_i \r)$, then
$\l\| \l(\sum_{i=1}^n \wdt G_i \wdt G_i^\ast\r)^{1/2} \r\|_{S_{2r}}^{2r}\leq \l\| \l(\sum_{i=1}^n \wdt G_i^\ast \wdt G_i\r)^{1/2} \r\|_{S_{2r}}^{2r}$, which implies that 
$\tr\l( \l(\sum_{i=1}^n \wdt G_i \wdt G_i^\ast \r)^r \r)\leq \tr\l( \l(\sum_{i=1}^n \wdt G_i^\ast \wdt G_i\r)^r \r)$, and
\begin{align*}
&
 \l\| \l(\sum_{i=1}^n Q_i^2\r)^{1/2}  \r\|_{S_{2r}}^{2r} 
 %= \tr\l( \l(\sum_{i=1}^n G_iG_i^T\r)^r \r) + \tr\l( \l(\sum_{i=1}^n G_i^TG_i\r)^r \r)
 \leq 2d\cdot\l\| \sum_{i=1}^n \wdt G_i^\ast \wdt G_i \r\|^r.
 \end{align*}
 
\item if $\l\| \sum_{i=1}^n \wdt G_i \wdt G_i^\ast  \r\| > \frac{1}{d}\tr\l( \sum_{i=1}^n \wdt G_i^\ast \wdt G_i \r)$, let $\lambda_j$ be the $j$-th eigenvalue of $\sum_{i=1}^n \wdt G_i \wdt G_i^\ast$, and note that
\begin{multline*}
d > \frac{\tr\l( \sum_{i=1}^n \wdt G_i^\ast \wdt G_i \r)}{\l\| \sum_{i=1}^n \wdt G_i \wdt G_i^\ast  \r\|} = 
\frac{\tr\l( \sum_{i=1}^n \wdt G_i \wdt G_i^\ast \r)}{\l\| \sum_{i=1}^n \wdt G_i \wdt G_i^\ast  \r\|} 
= \sum_{i=1}^{nd} \frac{\lambda_i}{\max_j \lambda_j}\geq \sum_{i=1}^{nd} \l(\frac{\lambda_i}{\max_j\lambda_j}\r)^r,\end{multline*}
where $r\geq 1$. In turn, it implies that
\begin{equation*}
\tr  \l(\l( \sum_{i=1}^n \wdt G_i \wdt G_i^\ast \r)^r\r) <
d \l\| \sum_{i=1}^n \wdt G_i \wdt G_i^\ast  \r\|^r. 
\end{equation*}
Thus
\begin{align*}
\l\| \l(\sum_{i=1}^n Q_i^2\r)^{1/2}  \r\|_{S_{2r}}^{2r} =& \tr\l( \l(\sum_{i=1}^n \wdt G_i \wdt G_i^\ast\r)^r \r) + \tr\l( \l(\sum_{i=1}^n \wdt G_i^\ast \wdt G_i\r)^r \r)\\
\leq& d \l\| \sum_{i=1}^n \wdt G_i \wdt G_i^\ast  \r\|^r + d \l\| \sum_{i=1}^n \wdt G_i^\ast \wdt G_i  \r\|^r.
%\leq& 2d\l( \l\| \sum_{i=1}^nG_iG_i^T  \r\|^r +  \l\| \sum_{i=1}^nG_i^TG_i  \r\|^r \r).
\end{align*}
\end{itemize}
Putting the bounds together, we obtain that
\begin{align}
\label{eq:f10}
\mathbb{E}_{\varepsilon} \l\| \sum_{i=1}^n\varepsilon_i Q_i^2 \r\|_{S_{2r}}^{2r}  \leq
&
2d C_r^{1/2} \max_i \l\| Q_i^2 \r\|^r  \max \l\{ \l\| \sum_{i=1}^n \wdt G_i \wdt G_i^\ast  \r\|^r, \l\| \sum_{i=1}^n \wdt G_i^\ast \wdt G_i  \r\|^r \r\} \\
\nonumber 
\leq & 
2d C_r^{1/2} \max_i \l\| Q_i^2 \r\|^r \cdot\l\| \sum_{i=1}^n Q_i^2  \r\|^r.
\end{align}
Next, observe that for $r$ such that $2r\geq q$, 
$\mb E_\eps \l\| \sum_{j=1}^n \eps_j Q_j^2 \r\|^{q}\leq \l( \mb E_\eps \l\| \sum_{j=1}^n \eps_j Q_j^2 \r\|^{2r}\r)^{q/2r}$ by H\"{o}lder's inequality, hence
\begin{align*}
\l( \mb E \l\| \sum_{j=1}^n \eps_j Q_j^2 \r\|^q \r)^{1/2q} &
= \l( \mb E \mb E_\eps \l\| \sum_{j=1}^n \eps_j Q_j^2 \r\|^{q} \r)^{1/2q} 
\\
& \leq \l( \mb E \l( \mb E_\eps \l\| \sum_{j=1}^n \eps_j Q_j^2 \r\|^{2r} \r)^{q/2r} \r)^{1/2q}
\leq \l( \mb E \l( \mb E_\eps \l\| \sum_{j=1}^n \eps_j Q_j^2 \r\|_{S_{2r}}^{2r} \r)^{q/2r} \r)^{1/2q}
\\
&\leq \l( 2d C_r^{1/2}\r)^{1/4r} \l( \mb E  \l[ \max_i \l\| Q_i^2 \r\|^{q/2} \cdot\l\| \sum_{i=1}^nQ_i^2  \r\|^{q/2}\r] \r)^{1/2q}.
\end{align*}
Set $r = q\vee\log d$ and apply Cauchy-Schwarz inequality to deduce that
\begin{align}
\label{eq:f20}
&
\l( \mb E \l\| \sum_{j=1}^n \eps_j Q_j^2 \r\|^q \r)^{1/2q} \leq
(8r)^{1/4}
\l(\mb E \max_i \l\| Q_i^2 \r\|^{q}\r)^{1/(4q)} \l( \mb E\l\| \sum_{i=1}^n Q_i^2  \r\|^{q}\r)^{1/(4q)}.
\end{align}
Substituting bound \eqref{eq:f20} into \eqref{inter-chaos-3} and letting
\[
R_q := \l( \mb E\l\| \sum_{i=1}^n Q_i^2  \r\|^{q}\r)^{1/(2q)},
\]
we obtain
\[
R_q \leq (8r)^{1/4}\sqrt{2R_q} \, \l( \mb E\max_i \l\| Q_i^2 \r\|^{q}\r)^{1/(4q)} + 
\l( \mb E \l\| \sum_{i=1}^n \mb E_2 Q_i^2 \r\|^{q}  \r)^{1/2q}.
\]
If $x,a,b>0$ are such that $x\leq a\sqrt{x}+b$, then $x\leq 4a^2 \vee 2b$, hence
\[
R_q\leq 16\sqrt{2r} \l( \mb E\max_i \l\| Q_i^2 \r\|^{q}\r)^{1/(2q)} + 
2\l( \mb E \l\| \sum_{i=1}^n \mb E_2 Q_i^2 \r\|^{q}  \r)^{1/2q}.
\]
Finally, it follows from \eqref{inter-chaos-2} that
\begin{align}
\label{eq:d10}
\nonumber
\Bigg( \mb E\Bigg\| \sum_{(i_1,i_2)\in I_n^2}\varepsilon_{i_1}^{(1)}\varepsilon_{i_2}^{(2)} & H(X_{i_1}^{(1)},X_{i_2}^{(2)}) \Bigg\|^{2q}\Bigg)^{1/(2q)} \\
\leq \nonumber
64\sqrt{\frac{2}{e}} \, r^{3/2} & \l( \mb E\max_i \l\| Q_i^2 \r\|^{q}\r)^{1/(2q)} + 
\frac{8}{\sqrt{e}} r \l( \mb E \l\| \sum_{i=1}^n \mb E_2 Q_i^2 \r\|^{q}  \r)^{1/2q}
\\ \nonumber
\leq 
64\sqrt{\frac{2}{e}} \, r^{3/2} & \l(\mb E \max_i \l\| \wdt G_i^\ast \wdt G_i \r\|^q \r)^{1/(2q)} 
+\frac{8}{\sqrt{e}} r \l( \mb E\l\|  \sum_{i=1}^n \mb E_2 \wdt G_i \wdt G_i^\ast \r\|^q + \mb E  \l\|  \sum_{i=1}^n \mb E_2 \wdt G_i^\ast \wdt G_i \r\|^{q} \r)^{1/2q}
\\ \nonumber
=
64\sqrt{\frac{2}{e}} r^{3/2} & \l( \mb E\max_{i_1} \l\| \sum_{i_2:i_2\ne i_1} H_{i_1,i_2}^2\l(X_{i_1}^{(1)},X_{i_2}^{(2)}\r) \r\|^q \r)^{1/(2q)}
\\
&
+ \frac{8}{\sqrt{e}} r \l( \mb E\l\|  \sum_{i=1}^n \mb E_2 \wdt G_i \wdt G_i^\ast \r\|^{q} +  
\mb E\l\|  \sum_{i_1=1}^n \l( \sum_{i_2:i_2\ne i_1} \mb E_2H_{i_1,i_2}^2\l(X_{i_1}^{(1)},X_{i_2}^{(2)}\r) \r)\r\|^{q} \r)^{1/2q},
\end{align}
where the last equality follows from the definition of $\wdt G_i$. 
To bring the bound to its final form, we will apply Rosenthal's inequality (Lemma \ref{lemma:rosenthal-pd}) to the last term in \eqref{eq:d10} to get that
\begin{multline*}
\l(\mb E\l\|  \sum_{i_1=1}^n \l( \sum_{i_2:i_2\ne i_1} \mb E_2H_{i_1,i_2}^2\l(X_{i_1}^{(1)},X_{i_2}^{(2)}\r) \r) \r\|^{q} \r)^{1/2q}
\leq \l\| \sum_{(i_1, i_2)\in I_n^2} \mb EH_{i_1,i_2}^2\l(X_{i_1}^{(1)},X_{i_2}^{(2)}\r) \r\|^{1/2}  
\\
+2\sqrt{2er} \l( \mb E \max_{i_1} \l\| \sum_{i_2:i_2\ne i_1} \mb E_2 H_{i_1,i_2}^2\l(X_{i_1}^{(1)},X_{i_2}^{(2)}\r) \r\|^q \r)^{1/2q}.
\end{multline*}
Moreover, Jensen's inequality implies that
\begin{align*}
&
\mb E \max_{i_1} \l\| \sum_{i_2:i_2\ne i_1} \mb E_2 H_{i_1,i_2}^2\l(X_{i_1}^{(1)},X_{i_2}^{(2)}\r) \r\|^q 
 \leq 
 \mb E \max_{i_1} \l\| \sum_{i_2:i_2\ne i_1} H_{i_1,i_2}^2\l(X_{i_1}^{(1)},X_{i_2}^{(2)} \r)\r\|^q,
\end{align*}
hence this term can be combined with one of the terms in \eqref{eq:d10}. 

%###################################################
\subsubsection{Proof of Lemma \ref{lemma:degen-lower-bound}.}
\label{proof:lower}
%###################################################

Let $\mb E_i[\cdot]$ stand for the expectation with respect to the variables with the upper index $i$ only. 
Since $H_{i_1,i_2}\l(\cdot, \cdot\r)$ are permutation-symmetric, we can apply the second part of Lemma \ref{decoupling-lemma} and (twice) the desymmetrization inequality (see Theorem 3.1.21 in \cite{gine2016mathematical}) to get that for some absolute constant $C_0>0$
\begin{align*}
%n(n-1)
\l( \mb E\l\| U_n \r\|^{2q} \r)^{1/(2q)}\geq 
& \frac{1}{C_0} \l( \mb E\l\| \sum_{(i_1,i_2)\in I_n^2} H_{i_1,i_2}\l(X_{i_1}^{(1)}, X_{i_2}^{(2)}\r) \r\|^{2q} \r)^{1/(2q)}
\\
=& \frac{1}{C_0} \l( \mb E_2\,\mathbb{E}_1 \l\|  \sum_{i_1}\sum_{i_2:i_2\ne i_1} H_{i_1,i_2}\l(X_{i_1}^{(1)}, X_{i_2}^{(2)}\r) \r\|^{2q} \r)^{1/(2q)}
\\
\geq& 
\frac{1}{2C_0} \l( \mb E\l\|  \sum_{i_1}\varepsilon_{i_1}^{(1)}\sum_{i_2:i_2\ne i_1} H_{i_1,i_2}\l(X_{i_1}^{(1)}, X_{i_2}^{(2)}\r) \r\|^{2q} \r)^{1/(2q)} 
\\
=&  \frac{1}{2C_0}\l( \mb E\mathbb{E}_2 \l\|  \sum_{i_2}\sum_{i_1:i_1\ne i_2}\varepsilon_{i_1}^{(1)} H_{i_1,i_2}\l(X_{i_1}^{(1)}, X_{i_2}^{(2)}\r) \r\|^{2q} \r)^{1/(2q)} 
\\
\geq&  
\frac{1}{4C_0} \l( \mb E\l\|  \sum_{(i_1,i_2)\in I_n^2}\varepsilon_{i_1}^{(1)}\varepsilon_{i_2}^{(2)} H_{i_1,i_2}\l(X_{i_1}^{(1)}, X_{i_2}^{(2)}\r) \r\|^{2q} \r)^{1/(2q)}.
\end{align*}
Applying the lower bound of Lemma \ref{bound-on-expectation} conditionally on $\l\{ X_{i}^{(1)} \r\}_{i=1}^n$ and $\l\{ X_{i}^{(2)} \r\}_{i=1}^n$, we obtain
\begin{align}
%n(n-1)
\l( \mb E\l\| U_n \r\|^{2q}\r)^{1/(2q)}
\geq&
c \l(\mb E\max\l\{ \l\| \sum_i \wdt G_i^\ast \wdt G_i \r\|, \l\| \sum_i \wdt G_i \wdt G_i^\ast \r\| \r\}^{q} \r)^{1/2q} 
\label{lower-bound-1}\\
\geq&\frac{1}{4\sqrt 2C_0} \l( \l( \mb E\l\|  \sum_{i} \wdt G_i \wdt G_i^\ast \r\|^{q} \r)^{1/2q} +  
\l( \mb E\l\|  \sum_{i} \wdt G_i^\ast \wdt G_i\r\|^{q} \r)^{1/2q} \r)
\nonumber \\
\geq & \frac{1}{4\sqrt{2} C_0} \l( \l( \mb E\l\|  \sum_{i} \mb E_2 \wdt G_i \wdt G_i^\ast \r\|^{q}\r)^{1/2q} +  
\l( \mb E\l\|  \sum_{i}\mb E_2 \wdt G_i^\ast \wdt G_i\r\|^{q} \r)^{1/2q} \r),
\nonumber
\end{align}
where $\wdt G_i$ is the $i$-th column if the matrix $\wdt G$ defined in \eqref{eq:g}; we also used the identities 
$\wdt G \wdt G^\ast = \sum_{i=1}^n \wdt G_i \wdt G_i^\ast,
~~\sum_{(i_1,i_2)\in I_n^2} H_{i_1,i_2}^2(X_{i_1}^{(1)},X_{i_2}^{(2)}) = \sum_{i=1}^n \wdt G_i^\ast \wdt G_i$. 
The inequality above takes care of the second and third terms in the lower bound of the lemma. 
To show that the first term is necessary, let
\[
Q_i = 
\left(
\begin{array}{ccc}
0 & \wdt G_i^\ast     \\
\wdt G_i & 0   \\   
\end{array}
\right).
\]
It follows from the first line of \eqref{lower-bound-1} that
\begin{align*}
\l( \mb E\l\| U_n \r\|^{2q} \r)^{1/(2q)}
\geq&
\frac{1}{4 C_0} \l( \mb E\l\| \sum_{i=1}^n Q_i^2 \r\|^{q} \r)^{1/(2q)}.
\end{align*}
Let $i_\ast$ be the smallest value of $i\leq n$ where $\max_i \l\| Q_i^2 \r\|$ is achieved. 
Then
$
\sum_{i=1}^n Q_i^2  \succeq Q_{i_\ast}^2,
$
hence $\l\|Q_{i_\ast}^2\r\|\leq \l\|\sum_{i=1}^n Q_i^2\r\|$. 
Jensen's inequality implies that
\begin{align*}
\l( \mb E\l\| U_n \r\|^{2q} \r)^{1/(2q)}
&\geq \frac{1}{4 C_0}\l( \mb E\l\| \sum_{i=1}^n Q_i^2 \r\|^{q} \r)^{1/(2q)}
\\
&\geq
\frac{1}{4 C_0} \l( \mb E\max_i \l\|Q_i^2\r\|^q \r)^{1/(2q)}
\geq 
\frac{1}{4 C_0} \l( \mb E \max_{i}\l\| \wdt G_i^\ast \wdt G_i \r\|^q \r)^{1/2q},
\end{align*}
where the last equality holds since $\l\| \wdt G_i^\ast \wdt G_i \r\| = \l\| \wdt G_i \wdt G_i^\ast \r\|$. 
The claim follows.

%###################################################
\subsubsection{Proof of Lemma \ref{lemma:remainder}.}
\label{proof:extraterm}
%###################################################

Note that 
\begin{multline}
\label{eq:app20}
r^{3/2} \l( \mb E\max_{i_1} \l\| \sum_{i_2:i_2\ne i_1} H_{i_1,i_2}^2 \l(X_{i_1}^{(1)},X_{i_2}^{(2)}\r) \r\|^q \r)^{1/(2q)} 
\\
\leq 
r^{3/2} \l( \mb E\sum_{i_1} \l\| \sum_{i_2:i_2\ne i_1} H_{i_1,i_2}^2 \l(X_{i_1}^{(1)},X_{i_2}^{(2)}\r) \r\|^q \r)^{1/(2q)}  
\\
=r^{3/2} \l( \mb E_1\sum_{i_1} \mb E_2 \l\| \sum_{i_2:i_2\ne i_1} H_{i_1,i_2}^2 \l(X_{i_1}^{(1)},X_{i_2}^{(2)}\r) \r\|^q \r)^{1/(2q)}.
\end{multline}
Next, Lemma \ref{lemma:rosenthal-pd} implies that, for $r=\max(q,\log(d))$, 
\begin{align}
\label{eq:app25}
\mb E_2 \l\| \sum_{i_2:i_2\ne i_1} H_{i_1,i_2}^2 \l(X_{i_1}^{(1)},X_{i_2}^{(2)}\r) \r\|^q 
\leq 
2^{2q-1} \Bigg[ & \l\| \sum_{i_2:i_2\ne i_1} \mb E_2 H_{i_1,i_2}^2 \l(X_{i_1}^{(1)},X_{i_2}^{(2)}\r) \r\|^q \\ 
& \nonumber
+(2\sqrt{2e})^{2q} r^{q} \, \mb E_2 \max_{i_2:i_2\ne i_1} \l\| H^2 \l(X_{i_1}^{(1)},X_{i_2}^{(2)}\r)\r\|^q  \Bigg].
\end{align}
We will now apply Lemma \ref{lemma:sum-max} with $\alpha=1$ and 
$\xi_{i_1}:= \l\| \sum_{i_2\ne i_1} \mb E_2 H^2 \l(X_{i_1}^{(1)},X_{i_2}^{(2)}\r) \r\|$ to get that
\begin{multline}
\label{eq:app30}
\sum_{i_1}\mb E\xi^q_{i_1} \leq 
2(1 + q) \Bigg( \mb E\max_{i_1}\l\| \sum_{i_2: i_2\ne i_1} \mb E_2 H_{i_1,i_2}^2 \l(X_{i_1}^{(1)},X_{i_2}^{(2)}\r) \r\|^q 
\\
+q^{-q} \l( \sum_{i_1} \mb E\l\| \sum_{i_2:i_2\ne i_1} \mb E_2 H_{i_1,i_2}^2 \l(X_{i_1}^{(1)},X_{i_2}^{(2)}\r) \r\| \r)^q \Bigg).
\end{multline}
Combining \eqref{eq:app20} with \eqref{eq:app25} and \eqref{eq:app30}, we obtain (using the inequality $1+q\leq e^q$) that
\begin{multline}
\label{eq:app40}
r^{3/2} \l( \mb E\max_{i_1} \l\| \sum_{i_2:i_2\ne i_1} H_{i_1,i_2}^2 \l(X_{i_1}^{(1)},X_{i_2}^{(2)}\r) \r\|^q \r)^{1/(2q)}
\\
\leq 4e\sqrt{2} \Bigg[ r^{3/2}\l(  \mb E\max_{i_1}\l\|  \sum_{i_2: i_2\ne i_1}  \mb E_2 H_{i_1,i_2}^2 \l(X_{i_1}^{(1)},X_{2}^{(2)}\r) \r\|^q \r)^{1/2q} \\
+ r \sqrt{1 + \frac{\log d}{q}}\l( \sum_{i_1} \mb E \l\|  \sum_{i_2: i_2\ne i_1} \mb E_2 H_{i_1,i_2}^2\l(X_{i_1}^{(1)},X_{i_2}^{(2)}\r) \r\|\r)^{1/2}
\\
+ r^2 \l( \sum_{i_1}\mb E\max_{i_2:i_2\ne i_1} \l\|  H_{i_1,i_2}^2 \l(X_{i_1}^{(1)},X_{i_2}^{(2)}\r) \r\|^q \r)^{1/2q}
\Bigg],
\end{multline}
which yields the result.
%first part of the claim. 

%#############################################
\subsection{Proof of Theorem \ref{thm:adamczak}.}
\label{section:adamczak-calculations}
%#############################################

Let $J\subseteq I\subseteq \{1,2\}$. 
We will write $\mf i$ to denote the multi-index $(i_1,i_2)\in \{1,\ldots,n\}^2$. 
We will also let $\mf i_I$ be the restriction of $\mf i$ onto its coordinates indexed by $I$, and, for a fixed value of 
$\mf i_{I^c}$, let $\l( H_\mf i \r)_{\mf i_I}$ be the array $\l\{ H_\mf i, \ \mf i_I \in \{1,\ldots,n\}^{| I |} \r\}$, where 
$H_{\mf i}:=H_{i_1,i_2}(X^{(1)}_{i_1},X^{(2)}_{i_2})$. 
Finally, we let $\mb E_I$ stand for the expectation with respect to the variables with upper indices contained in $I$ only. 
Following section 2 in \cite{adamczak2006moment}, we define 
\begin{align}
\label{eq:adam-norm}
\nonumber
\l\| \l( H_\mf i \r)_{\mf i_I}\r\|_{I,J} = \mb E_{I\setminus J}
\sup & \l\{ \mb E_J \sum_{\mf i_I}\langle \Phi, H_\mf i\rangle \prod_{j\in J} f^{(j)}_{ i_j}(X^{(j)}_{ i_j}): \ \|\Phi\|_\ast\leq 1,  \r.
\\
& 
\l. f_i^{(j)}:S\mapsto \mb R \text{ for all } i,j, \text{ and } 
\sum_i \mb E \l| f_i^{(j)}(X_i^{(j)})\r|^2 \leq 1, \ j\in J \r\}
\end{align}
and $\big\| \l( H_\mf i \r)_{\mf i_\emptyset}\big\|_{\emptyset,\emptyset}:=\l\| H_\mf i\r\|$, 
where $\langle A_1,A_2\rangle:=\tr(A_1\,A_2^\ast)$ for $A_1,A_2\in \mb H^d$ and $\|\cdot\|_\ast$ denotes the nuclear norm. 
Theorem 1 in \cite{adamczak2006moment} states that for all $q\geq 1$,
\begin{align*}
&
\l( \mb E \l\| U_n \r\|^{2q} \r)^{1/2q}\leq C \l[ \sum_{I\subseteq \{1,2\} }\sum_{J\subseteq I} 
q^{|J|/2 + | I^c |}  \l( \sum_{\mf i_{I^c}} \mb E_{I^c} \l\| \l( H_\mf i \r)_{\mf i_I}\r\|^{2q}_{I,J} \r)^{1/2q}
\r],
\end{align*}
where $C$ is an absolute constant. Obtaining upper bounds for each term in the sum above, we get that 
\begin{align*}
&
\l( \mb E \l\| U_n \r\|^{2q} \r)^{1/2q}\leq C\, \Big[ \mb E\l\| U_n \r\| + 
\sqrt{q} \cdot A + q\cdot B + q^{3/2}\cdot \Gamma + q^2 \cdot D \Big],
\end{align*}
where
\begin{align*}
A \leq & 2 \, \mb E_1 \l(\sup_{\Phi:\|\Phi\|_\ast\leq 1} \sum_{i_2} \mb E_2 \l\langle \sum_{i_1} H_{i_1,i_2}(X^{(1)}_{i_1},X^{(2)}_{i_2}), \Phi \r\rangle^2  \r)^{1/2}, \\
B \leq & \l(  \sup_{\Phi:\|\Phi\|_\ast\leq 1}  \sum_{(i_1, i_2) \in I_n^2} \mb E\l\langle H_{i_1,i_2}(X^{(1)}_{i_1},X^{(2)}_{i_2}), \Phi \r\rangle^2\r)^{1/2} \\
&+ 2\l( \sum_{i_2} \mb E_2 \l( \mb E_1 \sup_{\Phi:\|\Phi\|_\ast\leq 1} \l\langle \sum_{i_1} H_{i_1,i_2}(X^{(1)}_{i_1},X^{(2)}_{i_2}), \Phi  \r\rangle\r)^{2q}\r)^{1/2q}, \\
\Gamma \leq & 2 \l( \sum_{i_2} \mb E_2 \l( \sup_{\Phi:\|\Phi\|_\ast\leq 1} \sum_{i_1} \mb E_1 \l\langle H_{i_1,i_2}(X^{(1)}_{i_1},X^{(2)}_{i_2}),\Phi \r\rangle^2\r)^{q} \r)^{1/2q}, \\
D \leq & \l( \sum_{(i_1, i_2) \in I_n^2} \mb E \l\| H_{i_1,i_2}(X^{(1)}_{i_1},X^{(2)}_{i_2}) \r\|^{2q}\r)^{1/2q}.
\end{align*}
The bounds for $A,B,\Gamma,D$ above are obtained from \eqref{eq:adam-norm} via the Cauchy-Schwarz inequality. 
For instance, to get a bound for $A$, note that it corresponds to the choice $I = \{1,2\}$ and $J=\{1\}$ or $J=\{2\}$. Due to symmetry of the kernels, it suffices to consider the case $J=\{2\}$, and multiply the upper bound by a factor of $2$. 
When $J=\{2\}$, 
\begin{multline*}
\l( \sum_{\mf i_{I^c}} \mb E_{I^c} \l\| \l( H_\mf i \r)_{\mf i_I}\r\|^{2q}_{I,J} \r)^{1/2q} = 
\l\| \l( H_\mf i \r)_{\mf i_{\{1,2\}}}\r\|_{\{1,2\},\{2\}} 
\\
= \mb E_1 \sup\l\{ \mb E_2 \sum_{(i_1,i_2)\in I_n^2 } \langle H_{i_1,i_2}(X^{(1)}_{i_1},X^{(2)}_{i_2}),\Phi\rangle \cdot f^{(2)}_{i_2}\l( X^{(2)}_{i_2}\r) : \|\Phi\|_\ast \leq 1, \ \sum_{i_2} \mb E\l|f^{(2)}_{i_2}\l( X^{(2)}_{i_2}\r)\r|^2 \leq 1 \r\} 
\\
= \mb E_1\sup\l\{ \mb E_2 \sum_{i_2} \l\langle \sum_{i_1} H_{i_1,i_2}(X^{(1)}_{i_1},X^{(2)}_{i_2}),\Phi \r\rangle \cdot f^{(2)}_{i_2}\l( X^{(2)}_{i_2}\r) : \|\Phi\|_\ast \leq 1, \ \sum_{i_2} \mb E\l|f^{(2)}_{i_2}\l( X^{(2)}_{i_2}\r)\r|^2 \leq 1  \r\}
\\
\leq \mb E_1 \sup_{\Phi, f^{(2)}_{1},\ldots, f^{(2)}_{n} }\l\{ \sum_{i_2} \sqrt{ \mb E_2 \l\langle \sum_{i_1} H_{i_1,i_2}(X^{(1)}_{i_1},X^{(2)}_{i_2}),\Phi \r\rangle^2 }
\sqrt{\mb E  \l|f^{(2)}_{i_2}\l( X^{(2)}_{i_2}\r)\r|^2} \r\} 
\\
\leq \mb E_1 \sup \l\{ \l( \sum_{i_2} \mb E_2  \l\langle \sum_{i_1} H_{i_1,i_2}(X^{(1)}_{i_1},X^{(2)}_{i_2}),\Phi \r\rangle^2 \r)^{1/2} : \|\Phi\|_\ast \leq 1\r\} 
\\
= \mb E_1 \l(\sup_{\Phi:\|\Phi\|_\ast\leq 1} \sum_{i_2} \mb E_2 \l\langle \sum_{i_1} H_{i_1,i_2}(X^{(1)}_{i_1},X^{(2)}_{i_2}), \Phi \r\rangle^2  \r)^{1/2}.
\end{multline*}
It is not hard to see that the inequality above is in fact an equality, and it is attained by setting, for every fixed $\Phi$, 
\[
f^{(2)}_{i_2}\l( X_{i_2}^{(2)}\r) = \alpha_{i_2}\frac{ \l\langle \sum_{i_1} H_{i_1,i_2}(X^{(1)}_{i_1},X^{(2)}_{i_2}),\Phi \r\rangle}
{ \sqrt{ \mb E_2 \l\langle \sum_{i_1} H_{i_1,i_2}(X^{(1)}_{i_1},X^{(2)}_{i_2}),\Phi \r\rangle^2} },
\]
where $\alpha_{i_2} = \frac{ \sqrt{\mb E_2  \l\langle \sum_{i_1} H_{i_1,i_2}(X^{(1)}_{i_1},X^{(2)}_{i_2}),\Phi \r\rangle^2} }
{\sqrt{ \sum_{i_2}  \mb E_2  \l\langle \sum_{i_1} H_{i_1,i_2}(X^{(1)}_{i_1},X^{(2)}_{i_2}),\Phi \r\rangle^2 }  }$ are such that $\sum_{i_2} \alpha^2_{i_2}=1$. 
The bounds for other terms are obtained quite similarly. 
Next, we will further simplify the upper bounds for $A,B,\Gamma,D$ by analyzing the supremum over $\Phi$ with nuclear norm not exceeding $1$. To this end, note that 
\[
\Phi\mapsto \mb E\l\langle H_{i_1,i_2}(X^{(1)}_{i_1},X^{(2)}_{i_2}), \Phi \r\rangle^2
\] 
is a convex function, hence its maximum over the convex set $\{\Phi\in\mb H^d: \  \|\Phi \|_\ast \leq 1\}$ is attained at an extreme point that in the case of a unit ball for the nuclear norm must be a rank-1 matrix of the form $\phi \phi^\ast$ for some $\phi\in \mb C^d$. 
It implies that 
\begin{align}
\label{eq:app70}\nonumber
\sup_{\Phi:\|\Phi\|_\ast\leq 1} \mb E\l\langle H_{i_1,i_2}(X^{(1)}_{1},X^{(2)}_{2}), \Phi \r\rangle^2 
&\leq \sup_{\phi:\|\phi\|_2\leq 1} \mb E \l\langle H_{i_1,i_2}(X^{(1)}_{1},X^{(2)}_{2}), \phi \phi^\ast \r\rangle^2 \\
&
\leq \l\| \mb E H_{i_1,i_2}^2(X^{(1)}_{1},X^{(2)}_{2})\r\|.
\end{align}
Moreover, 
\begin{multline}
\label{eq:app80}
\sum_{i_2} \mb E_2 \l( \mb E_1 \sup_{\Phi:\|\Phi\|_\ast\leq 1} \l\langle \sum_{i_1} H_{i_1,i_2}(X^{(1)}_{i_1},X^{(2)}_{i_2}), \Phi  \r\rangle\r)^{2q} 
\\ 
= \sum_{i_2} \mb E_2 \l( \mb E_1 \l\| \sum_{i_1} H_{i_1,i_2}\l(X^{(1)}_{i_1},X^{(2)}_{i_2}\r) \r\| \r)^{2q} 
\\
\leq \sum_{i_2} \mb E_2\, \mb E_1  \l\| \sum_{i_1} H_{i_1,i_2}\l(X^{(1)}_{i_1},X^{(2)}_{i_2}\r) \r\|^{2q} 
\\
\leq 
\sum_{i_2} \mb E_2 \Bigg( 2\sqrt{er}\l\| \sum_{i_1} \mb E_1  H_{i_1,i_2}^2\l(X^{(1)}_{i_1},X^{(2)}_{i_2}\r) \r\|^{1/2} 
\\
+4\sqrt{2} e r \l( \mb E_1 \max_{i_1} \l\| H_{i_1,i_2}\l(X^{(1)}_{i_1},X^{(2)}_{i_2}\r)\r\|^{2q} \r)^{1/2q}  \Bigg)^{2q},
\end{multline}
where we have used Lemma \ref{lemma:rosenthal-1} in the last step, and $r=q\vee \log d$. 
Combining \eqref{eq:app70},\eqref{eq:app80}, we get that 
\begin{multline}
\label{eq:app90}
B\leq  \l( \l\| \sum_{(i_1, i_2) \in I_n^2} \mb E H_{i_1,i_2}^2(X^{(1)}_{i_1},X^{(2)}_{i_2})\r\| \r)^{1/2} 
+4\sqrt{er} \l( \sum_{i_2}\, \mb E_2 \l\| \sum_{i_1} \mb E_1  H_{i_1,i_2}^2\l(X^{(1)}_{i_1},X^{(2)}_{i_2}\r) \r\|^{q} \r)^{1/2q} \\
+ 8\sqrt{2}er \l( \sum_{i_2}\,\mb E \max_{i_1} \l\| H_{i_1,i_2}\l(X^{(1)}_{i_1},X^{(2)}_{i_2}\r)\r\|^{2q}\r)^{1/2q}.
\end{multline}
It is also easy to get the bound for $\Gamma$: first, recall that 
\[
\Phi\mapsto \sum_{i_1} \mb E_1 \l\langle H_{i_1,i_2}(X^{(1)}_{i_1},X^{(2)}_{i_2}),\Phi \r\rangle^2
\]
is a convex function, hence its maximum over the convex set $\{\Phi\in\mb H^d: \  \|\Phi \|_\ast \leq 1\}$ is attained at an extreme point of the form $\phi \phi^\ast$ for some unit vector $\phi$. 
Moreover, 
\begin{multline*}
\l\langle H_{i_1,i_2}(X^{(1)}_{i_1},X^{(2)}_{i_2}),\phi\phi^\ast \r\rangle^2 = 
\phi^\ast H_{i_1,i_2}(X^{(1)}_{i_1},X^{(2)}_{i_2})\phi \phi^\ast H_{i_1,i_2}(X^{(1)}_{i_1},X^{(2)}_{i_2}) \phi 
\\
\leq \phi^\ast H^2_{i_1,i_2}(X^{(1)}_{i_1},X^{(2)}_{i_2}) \phi 
\end{multline*}
due to the fact that $\phi\phi^\ast \preceq I$. Hence
\begin{multline*}
\sup_{\Phi:\|\Phi\|_\ast\leq 1} \sum_{i_1} \mb E_1 \l\langle H_{i_1,i_2}(X^{(1)}_{i_1},X^{(2)}_{i_2}),\Phi \r\rangle^2 \leq 
\sup_{\phi: \ \|\phi\|_2=1} \sum_{i_1} \mb E_1 \l(\phi^\ast H^2_{i_1,i_2}(X^{(1)}_{i_1},X^{(2)}_{i_2}) \phi \r)
\\
= \sup_{\phi: \ \|\phi\|_2=1}  \phi^\ast \l( \mb E_1 \sum_{i_1}  H^2_{i_1,i_2}(X^{(1)}_{i_1},X^{(2)}_{i_2}) \r)\phi 
= \l\| \mb E_1 \sum_{i_1}  H^2_{i_1,i_2}(X^{(1)}_{i_1},X^{(2)}_{i_2}) \r\|,
\end{multline*}
%\[
%\sup_{\Phi:\|\Phi\|_\ast\leq 1} \mb E_1 \l\langle H_{i_1,i_2}\l(X^{(1)}_{i_1},X^{(2)}_{i_2}\r),\Phi \r\rangle^2 
%\leq \l\| \mb E_1 H_{i_1,i_2}^2\l(X^{(1)}_{i_1},X^{(2)}_{i_2}\r) \r\|,  
%\]
and we conclude that
\begin{multline}
\label{eq:a110}
\Gamma \leq 2 \l( \sum_{i_2} \mb E_2 \l( \sup_{\Phi:\|\Phi\|_\ast\leq 1} \sum_{i_1} \mb E_1 \l\langle H_{i_1,i_2}(X^{(1)}_{i_1},X^{(2)}_{i_2}),\Phi \r\rangle^2\r)^{q} \r)^{1/2q}  
\\
\leq  2 \l( \sum_{i_2} \mb E_2 \l\| \sum_{i_1} \mb E_1 H_{i_1,i_2}^2\l(X^{(1)}_{i_1},X^{(2)}_{i_2}\r) \r\|^{q} \r)^{1/2q}.
\end{multline}
The bound for $A$ requires a bit more work. The following inequality holds:
\begin{lemma}
\label{lemma:A-bound}
The following inequality holds: 
\begin{multline*}
A\leq 2\mb E_1 \l(\sup_{\Phi:\|\Phi\|_\ast\leq 1} \sum_{i_2} \mb E_2 \l\langle \sum_{i_1} H_{i_1,i_2}(X^{(1)}_{i_1},X^{(2)}_{i_2}), \Phi \r\rangle^2  \r)^{1/2} \leq \\
64\sqrt{e}\log (de) \l( \mb E\max_{i_1} \l\| \sum_{i_2:i_2\ne i_1} H_{i_1,i_2}^2\l(X_{i_1}^{(1)},X_{i_2}^{(2)}\r) \r\|\r)^{1/2}
\\
+8\sqrt{2e\log (de)} \l( \mb E\l\|  \sum_{i} \mb E_2 \wdt G_i \wdt G_i^\ast \r\| +  
\l\|  \sum_{(i_1,i_2)\in I_n^2} \mb E H_{i_1,i_2}^2\l(X_{i_1}^{(1)},X_{i_2}^{(2)} \r)\r\| \r)^{1/2},
\end{multline*}
where $\wdt G$ was defined in \eqref{eq:g}.
\end{lemma}
Combining the bounds \eqref{eq:app90}, \eqref{eq:a110} and Lemma \ref{lemma:A-bound}, and grouping the terms with the same power of $q$, we get the result of Theorem \ref{thm:adamczak}.

It remains to prove Lemma \ref{lemma:A-bound}. 
To this end, note that Jensen's inequality and an argument similar to \eqref{eq:app70} imply that 
\begin{multline*}
\mb E_1 \l(\sup_{\Phi:\|\Phi\|_\ast\leq 1} \sum_{i_2} \mb E_2 \l\langle \sum_{i_1} H_{i_1,i_2}(X^{(1)}_{i_1},X^{(2)}_{i_2}), \Phi \r\rangle^2  \r)^{1/2} \\ 
\leq 
\l( \mb E \sup_{\Phi:\|\Phi\|_\ast\leq 1} \sum_{i_2} \l\langle \sum_{i_1} H_{i_1,i_2}(X^{(1)}_{i_1},X^{(2)}_{i_2}), \Phi \r\rangle^2  \r)^{1/2}
\\
\leq \l(\mb E\l\| \sum_{i_2} \l(\sum_{i_1} H_{i_1,i_2}(X^{(1)}_{i_1},X^{(2)}_{i_2})\r)^2 \r\| \r)^{1/2}.
\end{multline*}
Next, arguing as in the proof of Lemma \ref{bound-on-expectation}, we define 
\[
\widehat{B}_{i_1,i_2} = [0~|~0~|\ldots|~H_{i_1,i_2}(X^{(1)}_{i_1},X^{(2)}_{i_2})~|\ldots|~0~|~0]\in\mathbb{R}^{d\times nd},
\]
where $H_{i_1,i_2}$ sits on the $i_1$-th position of the block matrix above. Moreover, let 
\begin{align}
\label{eq:b}
B_{i_2}=\sum_{i_1:i_1\neq i_2}\widehat{B}_{i_1,i_2}.
\end{align}
Using the representation \eqref{eq:b}, we have
\begin{multline*}
\l(\mb E\l\| \sum_{i_2} \l(\sum_{i_1} H_{i_1,i_2}(X^{(1)}_{i_1},X^{(2)}_{i_2})\r)^2 \r\| \r)^{1/2}
= \l( \mb E\l\| \l(\sum_{i_2}B_{i_2}\r)\l(\sum_{i_2}B_{i_2}\r)^\ast \r\| \r)^{1/2} \\
=\l( \mb E
  \l\| \l( \sum_{i_2}\left(
\begin{array}{ccc}
0  &  B_{i_2}^\ast     \\
B_{i_2}  &   0   \\   
\end{array}
\right)  \r)^2\r\| \r)^{1/2}
\leq2\l( \mb E
  \l\| \l( \sum_{i_2}\varepsilon_{i_2}\left(
\begin{array}{ccc}
0  &  B_{i_2}^\ast     \\
B_{i_2}  &   0   \\   
\end{array}
\right)  \r)\r\|^2 \r)^{1/2},
\end{multline*}
where $\l\{\varepsilon_{i_2}\r\}_{i_2=1}^n$ is sequence of i.i.d. Rademacher random variables, and the last step follows from the symmetrization inequality. 
Next, Khintchine's inequality \eqref{matrix-Khintchine} yields that
\begin{align*}
A\leq& 4\sqrt{e(1+2\log d)}\l(\mb E\l\|\sum_{i_2}\left(
\begin{array}{ccc}
B_{i_2}^\ast B_{i_2}  & 0      \\
0  &  B_{i_2}B_{i_2}^\ast   \\   
\end{array}
\right)  \r\|\r)
^{1/2}\\
=&4\sqrt{e(1+2\log d)}\l(\mb E\max\l\{\l\| \sum_{i_2}B_{i_2}^\ast B_{i_2} \r\|,  \l\| \sum_{i_2}B_{i_2}B_{i_2}^\ast \r\|\r\}\r)^{1/2}
\\
=& 4\sqrt{e(1+2\log d)}\l( \mb E\max\l\{  \l\| \widetilde G\widetilde G^\ast\r\|, \left\|\sum_{(i_1,i_2)\in I^2_n}H_{i_1,i_2}^2\l( X^{(1)}_{i_1},X^{(2)}_{i_2} \r) \right\| \r\} \r)^{1/2}.
\end{align*}
Note that the last expression is of the same form as equation \eqref{inter-chaos-1} in the proof of Theorem \ref{thm:degen-moment} with $q=1$. Repeating the same argument, one can show that  
\begin{align*}
A \leq 64\sqrt{e}\log (de)& \l( \mb E\max_{i_1} \l\| \sum_{i_2:i_2\ne i_1} H_{i_1,i_2}^2\l(X_{i_1}^{(1)},X_{i_2}^{(2)}\r) \r\|\r)^{1/2}
\\
&
+8\sqrt{2e\log (de)} \l( \mb E\l\|  \sum_{i} \mb E_2 \wdt G_i \wdt G_i^\ast \r\| +  
\l\|  \sum_{(i_1,i_2)\in I_n^2} \mb E H_{i_1,i_2}^2\l(X_{i_1}^{(1)},X_{i_2}^{(2)} \r)\r\| \r)^{1/2},
\end{align*}
which is an analogue of \eqref{eq:d10}.

%#############################################
\subsection{Calculations related to Examples \ref{example:01} and \ref{example:02}.}
\label{section:example-proof}
%#############################################

We will first estimate $\l\|GG^\ast\r\|$. Note that the $(i,i)$-th block of the matrix $GG^\ast$ is
\[
\l(GG^\ast\r)_{ii} = \sum_{j:j\neq i}A_{i,j}^2 = \sum_{j:j\neq i}\l(\mathbf{a}_i\mathbf a_j^T+\mathbf a_j\mathbf a_i^T\r)^2
=(n-1)\mathbf a_i\mathbf a_i^T + \sum_{j\neq i}\mathbf a_j\mathbf a_j^T.
\]
The $(i,j)$-block for $j\neq i$ is
\[
\l(GG^\ast\r)_{ij} = \sum_{k\neq i,j}A_{i,k}A_{j,k}
=\sum_{k\neq i,j}\l(\mathbf a_i\mathbf a_k^T+\mathbf a_k\mathbf a_i^T\r)\l(\mathbf a_j\mathbf a_k^T+\mathbf a_k\mathbf a_j^T\r)=(n-2)\mathbf a_i\mathbf a_j^T.
\]
We thus obtain that
\begin{align*}
GG^\ast = (n-2)\mathbf{a}\mathbf{a}^T + \mathrm{Diag}\l( \underbrace{ \sum_{j=1}^n\mathbf{a}_j\mathbf{a}_j^T
,\ldots, \sum_{j=1}^n\mathbf{a}_j\mathbf{a}_j^T }_{\text{n terms}}\r),
\end{align*}
where  $\mathrm{Diag}(\cdot)$ denotes the block-diagonal matrix with diagonal blocks in the brackets.
Since 
$$\mathrm{Diag}\l( \sum_{j=1}^n\mathbf{a}_j\mathbf{a}_j^T
,\ldots, \sum_{j=1}^n\mathbf{a}_j\mathbf{a}_j^T\r)\succeq 0,$$ 
it follows that 
\[
\l\|GG^\ast \r\| \geq (n-2) \|\mathbf{a}\|^2_2=(n-2)n.
\]
On the other hand, 
\begin{multline*}
\left\| \sum_{(i_1,i_2)\in I_n^2} A_{i_1,i_2}^2\right\|
= \left\| \sum_{(i_1,i_2)\in I_n^2} \l(  \mathbf{a}_{i_1}\mathbf{a}_{i_2}^T + \mathbf{a}_{i_2}\mathbf{a}_{i_1}^T \r)^2\right\|
= \left\| \sum_{(i_1,i_2)\in I_n^2} \l( \mathbf{a}_{i_1}\mathbf{a}_{i_1}^T +\mathbf{a}_{i_2}\mathbf{a}_{i_2}^T  \r)  \right\| 
\\
=2(n-1)\l\| \sum_{i} \mathbf{a}_{i}\mathbf{a}_{i}^T \r\| = 2(n-1).
\end{multline*}
where the last equality follows from the fact that $\{\mathbf{a}_1,\ldots,\mathbf{a}_n\}$ are orthonormal. 

For Example \ref{example:02}, we similarly obtain that 
\begin{align*}
\l(GG^\ast\r)_{ii} &= \sum_{j:j\neq i}c^2_{i,j}\l(\mathbf{a}_i\mathbf a_j^T+\mathbf a_j\mathbf a_i^T\r)^2
\\
&=\l( \sum_{j:j\ne i} c^2_{i,j}\r)\mathbf a_i\mathbf a_i^T + \sum_{j:j\neq i} c^2_{i,j}\mathbf a_j\mathbf a_j^T 
= \mathbf a_i\mathbf a_i^T + \sum_{j:j\neq i} c^2_{i,j}\mathbf a_j\mathbf a_j^T,
\\
\l(GG^\ast\r)_{ij} &= \sum_{k\neq i,j}c_{i,k}c_{j,k}\l(\mathbf a_i\mathbf a_k^T+\mathbf a_k\mathbf a_i^T\r)\l(\mathbf a_j\mathbf a_k^T+\mathbf a_k\mathbf a_j^T\r)=0, \ i\ne j,
\end{align*}
hence $\l\| GG^\ast \r\|=\max_i \l\| (GG^\ast)_{ii}\r\|=1$. On the other hand, 
\begin{multline*}
\left\| \sum_{(i_1,i_2)\in I_n^2} A_{i_1,i_2}^2\right\|
= \left\| \sum_{(i_1,i_2)\in I_n^2} c^2_{i_1,i_2}\l(  \mathbf{a}_{i_1}\mathbf{a}_{i_2}^T + \mathbf{a}_{i_2}\mathbf{a}_{i_1}^T \r)^2\right\| \\
= \left\| \sum_{(i_1,i_2)\in I_n^2} c^2_{i_1,i_2}\l( \mathbf{a}_{i_1}\mathbf{a}_{i_1}^T +\mathbf{a}_{i_2}\mathbf{a}_{i_2}^T  \r)  \right\| = 2\l\| \sum_{i} \mathbf{a}_{i} \mathbf{a}_{i}^T \r\|=2,
\end{multline*}
and 
\begin{multline*}
\sum_{i_1} \left\| \sum_{i_2:i_2\neq i_1}A_{i_1,i_2}^2\right\| = \sum_{i_1} \left\| \sum_{i_2\neq i_1} c^2_{i_1,i_2}\l(  \mathbf{a}_{i_1}\mathbf{a}_{i_2}^T + \mathbf{a}_{i_2}\mathbf{a}_{i_1}^T \r)^2\right\| \\
= \sum_{i_1}\l\| \sum_{i_2:i_2\ne i_1} c_{i_1,i_2}^2 \l(  \mathbf{a}_{i_1} \mathbf{a}_{i_1}^T +  \mathbf{a}_{i_2}  \mathbf{a}_{i_2}^T\r) \r\| 
= \sum_{i_1}\l\|  \mathbf{a}_{i_1} \mathbf{a}_{i_1}^T  + \sum_{i_2:i_2\ne i_1} c_{i_1,i_2}^2   \mathbf{a}_{i_2}  \mathbf{a}_{i_2}^T \r\| = n.
\end{multline*}

%######################
%\section*{Acknowledgements}
%######################

%\bibliographystyle{amsplain}
%\bibliography{u-stat,u-stat-draft}

%\providecommand{\bysame}{\leavevmode\hbox to3em{\hrulefill}\thinspace}
%\providecommand{\MR}{\relax\ifhmode\unskip\space\fi MR }
%\providecommand{\MRhref}[2]{%
%  \href{http://www.ams.org/mathscinet-getitem?mr=#1}{#2}}
%\providecommand{\href}[2]{#2}

\end{document}

%% file: stas.tex
\newcommand{\dotp}[2]{\left\langle#1,#2\right\rangle}
\newcommand{\m}{\mathcal}
\newcommand{\mb}{\mathbb}

\newcommand{\mx}{\mbox{\footnotesize{max}\,}}
\newcommand{\mn}{\mbox{\footnotesize{min}\,}}

\def\r{\right}
\def\l{\left}
\newcommand{\eps}{\varepsilon}
\newcommand{\var}{\mbox{Var}}